\documentclass[a4paper,11pt]{article}

\usepackage{amsmath}
\usepackage{amssymb}
\usepackage{amsthm}
\usepackage{graphicx}
\usepackage{enumerate}
\usepackage[dvipsnames]{xcolor}
\usepackage{hyperref}
\newtheorem{theorem}{Theorem}[section]

\newtheorem{lemma}[theorem]{Lemma}

\newtheorem{assumption}{Assumption A}

\theoremstyle{definition}

\numberwithin{equation}{section}
\newcommand{\N}{\mathcal{N}}
\newcommand{\D}{\mathcal{D}}

\newcommand{\A}{\mathcal{A}}

\begin{document}
	
\title{AS-BOX: Additional Sampling Method for  Weighted Sum Problems with Box Constraints }
\author{
Nataša Krejić\thanks{Department of Mathematics and Informatics, Faculty of Sciences, University of Novi Sad, Trg Dositeja Obradovića 4, 21000 Novi Sad, Serbia. Emails: \texttt{natasak@uns.ac.rs}, \texttt{natasa.krklec@dmi.uns.ac.rs}}, 
Nataša Krklec Jerinkić\footnotemark[1], 
Tijana Ostojić\thanks{Department of Fundamental Sciences, Faculty of Technical Sciences, University of Novi Sad, Trg Dositeja Obradovića 6, 21000 Novi Sad, Serbia. Email: \texttt{tijana.ostojic@uns.ac.rs}}, 
Nemanja Vučićević\thanks{Department of Mathematics and Informatics, Faculty of Sciences, University of Kragujevac, Radoja Domanovića 12, 34000 Kragujevac, Serbia. Email: \texttt{nemanja.vucicevic@pmf.kg.ac.rs}.} \footnote{Corresponding author.}
}
%\author{Nata\v{s}a Kreji\'c\footnote{Department of Mathematics and Informatics, Faculty of Sciences, University of Novi Sad, Trg Dositeja Obradovi\' ca 4, 21000 Novi Sad, Serbia. e-mail: \texttt{natasak@uns.ac.rs}}, Nata\v{s}a Krklec Jerinki\'c \footnote{{\color{red} Mozemo li spojiti NK i NKJ?} Department of Mathematics and Informatics, Faculty of Sciences, University of Novi Sad, Trg Dositeja Obradovi\' ca 4, 21000 Novi Sad, Serbia. e-mail: \texttt{natasa.krklec@dmi.uns.ac.rs}},
%Tijana Ostoji\' c \footnote{Department of Fundamental Sciences, Faculty of Technical Sciences, University of Novi
%Sad, Trg Dositeja Obradovi\' ca 6, 21000 Novi Sad, Serbia. e-mail: \texttt{tijana.ostojic@uns.ac.rs}},
%Nemanja Vu\v{c}i\'cevi\'c\footnote{Department of Mathematics and Informatics, Faculty of Sciences, University of Kragujevac, Radoja Domanovi\' ca 12, 34000 Kragujevac, Serbia. e-mail: \texttt{nemanja.vucicevic@pmf.kg.ac.rs}} \footnote{Corresponding author}}
\date{November 25, 2025}
\maketitle

\begin{abstract}
A class of optimization problems characterized by a weighted finite-sum objective function subject to box constraints is considered. We propose a novel stochastic optimization method, named AS-BOX (\text{A}ddi\-ti\-onal \text{S}ampling for \text{BOX} constraints), that combines projected gradient directions with adaptive variable sample size strategies and nonmonotone line search. The method dynamically adjusts the batch size based on progress with respect to the additional sampling function and on structural consistency of the projected direction, enabling practical adaptivity of AS-BOX, while ensuring theoretical support.  We establish almost sure convergence under standard assumptions and provide complexity bounds. Numerical experiments demonstrate the efficiency and competitiveness of the proposed method compared to state-of-the-art algorithms.
\end{abstract}

\textbf{Key words:} 
 Projected Gradient Methods, Sample Average Approximation, Adaptive Variable Sample Size Strategies, Non-monotone Line Search, Additional Sampling, Almost Sure Convergence.

\textbf{MSC Classification:} 90C15,90C26, 90C30, 65K05
%90C15, 90C30, 65K05, 65K10. {\color{red} Ovo je isto kao za ASPEN?}
\section{Introduction}
\label{sec1}
 We consider a box-constrained optimization problem with the objective function in the form of a weighted finite sum, i.e., 
 \begin{equation} \label{problem}
\min_{x \in S} f(x):=\sum_{i=1}^{N} w_i f_i(x),\;   S= \{x \in \mathbb{R}^n \; | \; l_i \leq x_i \leq u_i, \; i=1,...,n\},
\end{equation}
where $S$ represents the feasible set defined by  
\begin{equation*} 
%\label{lu}
      l_i \in \mathbb{R}\cup \{-\infty\}, u_i \in \mathbb{R}\cup \{\infty\}, i=1,2,...,n,
\end{equation*}
while $ f_i:\mathbb{R}^n \to \mathbb{R}, i=1,...,N$ are continuously-differentiable functions and  $w_1,...,w_N$ represent the weights such that  
\begin{equation*} 
%\label{w}
    \sum_{i=1}^{N} w_i=1, \quad w_i\geq 0, \; i=1,..,N.
\end{equation*}
%Notice that this problem is a generalization of unconstrained finite sum problems, i.e., $w_i=1/N$ and $S=\mathbb{R}^n$. Moreover, it also covers the case of nonnegative constraints $x\geq 0$, i.e., $x_i \geq 0, i=1,...,n$.
This problem captures a wide class of practical problems. It generalizes classical unconstrained finite-sum formulations, where $w_i=1/N$ and $S=\mathbb{R}^n$ (e.g., empirical risk minimization with uniform weights) by allowing arbitrary positive weights and bound constraints. The formulation with unequal weights is motivated by the so-called local regression models (see e.g.,\cite{UMU}) where the weights represent the importance (distance) of different data points in the training set with respect to a new data point. However if uniform sampling is more convenient from practical (implementation) point of view, the weights can be integrated to form the local cost functions $\tilde{f}_i(x):=w_i f_i(x)$ and equivalent problem $\min_{x \in S} \frac{1}{N} \sum_{i=1}^{N} \tilde{f}_i(x)$ can be solved. Although applying the proposed algorithm would yield different iterations due to randomness and different subsampling distributions (see \eqref{Nk} ahead), theoretical guarantees remain the same as for the weighted sum case that we consider. Further,   special cases of the considered problem include nonnegative constraints $x\geq 0$, i.e., $x_i \geq 0, i=1,...,n$. Such box-constrained problems naturally arise in machine learning, signal processing, portfolio optimization, and computational statistics \cite{wright}. 

%{\bf{Motivation.}}  
In large-scale optimization problems of form \eqref{problem}, evaluating the full gradient $\nabla f(x) = \sum_{i=1}^N w_i \nabla f_i(x)$ can be computationally expensive, especially when $N$ is large. Stochastic methods based on subsampling are widely used to reduce this cost. One possibility to reduce the cost while maintaining reasonably good approximations of the gradients is to use adaptive sampling strategies. These strategies dynamically refine gradient approximations during the optimization process, mainly by progressive increase of the sample size based on variance estimates or structural indicators such as descent quality or direction stability. This way one is able to maintain a balance between computational efficiency and convergence reliability. Adaptive sampling methods have been extensively studied in recent years \cite{pregledni, SBNKNKJ, SBNKNKJMR, byrd_adap, iusem1, iusem2, nasprvi, proxbb}.
%For instance, in the context of unconstrained stochastic optimization, in \cite{lsnmbb} it is proposed a variable sample size quasi-Newton method that adjusts the sample size based on gradient disagreement. Similarly, the method with adaptive sample size introduced in \cite{proxbb} leverages the growth of the sample based on practical stopping tests, enabling significant computational savings.

One effective subclass of adaptive sampling is the so-called additional sampling, which typically increases the sample size when a prescribed criterion fails \cite{LSOS, krejic4, aspen,ipas,proxbb,lsnmbb}. The criterion of progress is defined by an additional sampling, i.e., a new independent subsample (of modest size) is generated after each iteration and is used to accept/reject the iteration.  As the additional sample is mainly of modest size, this approach avoids excessive computational cost while still ensures convergence in a stochastic sense.  For example, in the IPAS method \cite{ipas}, the additional sampling is combined with projected gradient steps for problems with linear equality constraints. Sampling growth is governed by a descent-based condition that assess whether the current sample is sufficient to ensure meaningful progress. This mechanism allows the method to operate efficiently in early iterations with small batches and to increase precision only when needed. Furthermore,
the ASPEN method \cite{aspen} extends this idea to nonlinear equality-constrained problems by incorporating a quadratic penalty term. In this method, additional sampling is applied adaptively based on indicators such as gradient norm and descent quality. The method attempts to make progress using the current sample and increases the sample size only when necessary, making it particularly effective near critical points where variance in gradient estimates becomes more pronounced. 

Another key component of the method proposed here is nonmonotone line search (NMLS). Classical Armijo-type condition requires a sufficient decrease in each iteration, which can be restrictive and lead to overly conservative steps in noisy settings. Nonmonotone line searches, on the other hand, allow temporary increases in the objective function, enabling better exploration of the landscape and improving practical performance. A number of NMLS is present in the literature,  \cite{grippo, hz, li} and they are successfully applied in numerous deterministic and stochastic frameworks (e.g., \cite{Birgin1, hu, krejic4, NKNKJ2,  nasdrugi, lsnmbb}). In stochastic settings, the effect of noise and variance by relaxing strict descent conditions are particularly important. The NMLS method we rely on is originally defined in \cite{li}. 
%When combined with adaptive sampling, NMLS leads to more flexible and efficient algorithms, as shown in \cite{aspen, ipas,proxbb}.

Interior-point methods (IPMs) represent another popular class of algorithms for constrained optimization, known for their strong theoretical properties and practical efficiency. Numerous works have extended IPMs to accommodate large-scale and structured problems, including both deterministic and stochastic settings. For instance, classical interior-point frameworks tailored for convex programming and barrier methods are well-established (e.g., \cite{ nesterov1994interior, wright1997primal}), while more recent advances incorporate stochastic elements or specific constraint structures (e.g., \cite{he2024, roos1997interior}).
In the stochastic optimization literature, interior-point methods have also been adapted to settings where exact gradients are either expensive or impossible to compute. These adaptations often involve inexact or sampled gradient approximations and the use of approximate barrier subproblems to preserve feasibility and convergence properties under uncertainty. A recent contribution in this direction is  \cite{curtis}, where a stochastic gradient-based interior point method (SIPM) to solve box-constrained optimization problems is proposed. 
%{\color{red} OVO JE TAJ DEO GDE SE SPOMINJAO RAD IZ NUMERICKIH REZULTATA.} 
%Their method integrates stochastic gradient information within an interior-point framework, addressing the challenges of constrained stochastic optimization by using approximate barrier subproblems and adaptive sampling techniques. The authors demonstrate strong theoretical convergence guarantees and provide empirical evidence of the algorithm's efficiency in handling large-scale problems where exact gradients are expensive or impractical to compute. 
The SIPM algorithm extends the classical interior-point framework to the stochastic setting by augmenting the objective with a logarithmic barrier that enforces box constraints and by employing a prescribed decreasing sequence of barrier parameters rather than adaptive updates. Unlike standard interior-point methods, SIPM maintains iterates within progressively shrinking inner neighborhoods of the feasible box and avoids fraction-to-the-boundary rules or line searches, which are challenging to implement in stochastic regimes. 

As the baseline algorithm, we use the Projected Stochastic Gradient Method (PSGM), which originates from the classical framework of projected gradient methods 
%\cite{Birgin1}, 
and is here adapted into a stochastic version following the implementation in \cite{curtis}. PSGM is a projection-based method that iteratively computes stochastic gradient steps on the original objective and projects them back onto the feasible region, i.e., updates are of the form \(x_{k+1} = \pi_{[l,u]}(x_k - \alpha_k g_k)\), where \(\pi_{[l,u]}\) denotes the projection onto the box constraints. While SIPM leverages barrier smoothing to handle boundaries implicitly, PSGM enforces feasibility explicitly through projection. These two methods are used for numerical comparison in this paper. 
%respectively, interior-point and projected-gradient paradigms for bound-constrained stochastic optimization. 

The method we propose here, AS-BOX is a novel stochastic optimization algorithm for weighted finite-sum problems with box constraints. Our key contributions include the following. A new stochastic method for solving the box constrained problems is proposed and analysed, both theoretically and numerically. The method relays on the well-established non-monotone line search along the projected subsampled gradient direction. The key innovation is additional sampling of modest size performed in each iteration, which yields two advantages. First, it resolved the theoretical issue of mutual dependence of the direction and stepsize and hence allowed us to prove a.s. convergence of the method. Second, the additional sampling results in a natural subsampling schedule, that is problem dependent (not predefined). The worst-case complexity is also analyzed, providing an expected number of iterations to reach the vicinity of the stationary points of the considered problem. Numerical results are presented on real-world data with logistic regression and Neural Network problems as test cases.

%\begin{itemize}
%\item[-] A  design of a novel first-order stochastic method based on projected gradient directions, non-monotone line search and additional sampling technique to guide the potential increase of the sample size. The proposed method can be seen as a natural extension of the existing adaptive sampling frameworks for equality-constrained problems \cite{aspen, ipas} to the box-constrained case. 
%, preserving their strengths while adapting to a specific constraint structure.
%\item[-] We incorporate an adaptive additional sampling strategy to balance convergence quality and computational efficiency.
%item[-] We utilize a robust nonmonotone line search mechanism to improve convergence in the presence of stochastic noise and boundary behavior.
%\item[-] The a.s. convergence of AS-BOX is proved under a set of standard assumptions. Moreover, some complexity bounds are provided for the proposed method as well.
%\item[-] Demonstration of the method's effectiveness in numerical experiments on real-world data, including both logistic regression and Neural Network problems. 
%\end{itemize}

{\bf{Paper organization.}} The paper is organized as follows. Section~\ref{sec2} provides the necessary preliminaries. In Section~\ref{sec3}, we present the AS-NC method designed for problems with non-negativity constraints, including the algorithmic framework and convergence analysis. Although non-negativity constraints are a special case of the general box constraints we start with this case for clarity of exposition. Then, in Section~\ref{sec4} we generalize to AS-BOX method, our main contribution, for solving box-constrained problems. Numerical experiments are reported in Section~\ref{sec5}, while Section~\ref{sec6} concludes the paper. \\
 
\textbf{Notation.} Throughout the paper, we use the following notation:
$\mathbb{R}_+$ denotes the set of non-negative real numbers.
The symbol $\|\cdot\|$ represents the standard Euclidean norm. The expectation operator is denoted by $\mathbb{E}(\cdot)$, and $\mathbb{E}(\cdot \mid \mathcal{F})$ stands for the conditional expectation given a $\sigma$-algebra $\mathcal{F}$.  We use “a.s.” to abbreviate “almost sure”. For a finite set $A,$ $|A|$ denotes its cardinality. 
%Finally, $\pi_S(x)$ denotes the orthogonal projection of a point $x$ onto the set $S$.

\section{Preliminaries}\label{sec2}
 Let us denote by $\pi_{S}(y)$ the orthogonal projection of a point $ y$ on the set $S$. 
 %It is easy to see that $[\pi_S(y)]_i$ is equal to: 1) $l_i$ if $y_i < l_i$; 2) $u_i$ if $y_i > u_i$; 3) $y_i$ if $l_i \leq y_i \leq u_i$, for $i=1,..,n$, i.e.,
%\begin{equation} 
%\label{pibox}
%[\pi_S(y)]_i = \min\left\{ \max\left\{ y_i,\, l_i \right\},\, u_i \right\}, \quad \text{for } i = 1, \dots, n.
%\end{equation}
One can show that the projected gradient direction of the form 
\begin{eqnarray}
    \label{pgd}
    d(x):=\pi_{S}(x-\nabla f(x))-x
\end{eqnarray}
is a descent direction for function $f$ at point $x \in S$ unless $x$ is a stationary point of problem \eqref{problem}. More precisely, the following result is known. 
\begin{theorem} \cite{Birgin1}
    \label{bmr} Assume that $f \in C^1(S_k)$ and $x \in S$. Then the projected gradient direction \eqref{pgd} satisfies: 
    \begin{itemize}
        \item[a)] $d^T(x) \nabla f(x) \leq -\|d(x)\|^2.$
        \item[b)] $d(x)=0$ if and only if $x$ is a stationary point of problem \eqref{problem}.
    \end{itemize}
\end{theorem}
We will be dealing with approximate evaluations of the objective function and its gradients. More precisely, we use the following sample-based estimate of the objective function at iteration $k$ in general \cite{ipas} 
\begin{equation} \label{fnk}
f_{\N_k}(x):=\frac{1}{N_k}\sum_{i\in \N_k}  f_i(x),
\end{equation}
where $N_k:=| \N_k|$, $\N_k=\{i^k_1,...,i^k_{N_k}\}$, and  each $i^k_j \in \N_k$ takes the value $s \in \N:=\{1,...,N\}$ with probability $w_s$, i.e., 
\begin{equation} \label{Nk} P(i^k_j=s)=w_s,\quad  s=1,...,N,\; j=1,...,N_k.
\end{equation}
 This way we have an unbiased estimate of $f$, i.e., 
$$\mathbb{E}(f_{\N_k}(x)|x)=\mathbb{E}(\frac{1}{N_k}\sum_{j=1}^{N_k}  f_{i^k_j}(x)|x)=\frac{1}{N_k}\sum_{j=1}^{N_k}  \mathbb{E}(f_{i^k_j}(x)|x)=\frac{1}{N_k}\sum_{j=1}^{N_k}  f(x)=f(x),$$
where $\mathbb{E}(\cdot|x)$ denotes conditional probability given the point $x$. However, this is not crucial for the analysis, and the convergence results hold for an arbitrary sampling of $\N_k$ as well. Moreover, since the method that will be proposed in the sequel may reach the full sample size, we will assume that when $N_k=N$ we simply take the whole sample, i.e., $\N_k=\N$. The approximate gradient will be taken as the gradient of the approximate function $\nabla f_{\N_k}$.

Since we work with approximate functions in general,  non-monotone Armijo-type line search will be employed \cite{li} to determine the step size $t_k$ given a direction $p_k= \pi_S(x_k-\nabla f_{\N_k}(x_k))-x_k$
\begin{equation*} 
%\label{nonm} 
f_{\N_k}(x_k+t_k p_k)\leq f_{\N_k}(x_k)  +
c_1 t_k (\nabla f_{\N_k}(x_k))^T p_k+\varepsilon_k,
\end{equation*}
where  $\varepsilon_k>0, k \in \mathbb{N}$ represents a predetermined sequence which satisfies the following condition
\begin{equation} 
\label{sumable}  
\sum_{k=0}^{\infty} \varepsilon_k \leq \bar{\varepsilon}< \infty.
\end{equation}
Notice that the search direction $p_k$ is a descent direction for the function $f_{\N_k}$ at point $x_k$. Moreover, $x_k+p_k$ is feasible provided that $x_k $ is feasible as well, and due to the convexity of $S$, backtracking line search will ensure that $x_k+t_k p_k $ remains in the feasible set. Thus, starting from $x_0 \in S$, the proposed algorithm will ensure the feasibility of all the iterates. 

We apply an additional sampling technique to guide the sample size increase. Additional sampling is used to overcome bias that comes from the dependency of the candidate iterate $\bar{x}_k=x_k+t_k p_k $ on the sample $\N_k$. Moreover, it can be viewed as a check on the similarity of the local cost functions - if they are heterogeneous, then it is probably beneficial to increase the sample size since the mini-batch estimate is not good enough representative of the objective function. For more details one can see \cite{ipas} and the references therein. We form an additional sampling function similarly to $f_{\N_k}$, but with a much smaller sample in general. Namely, 
we have 
\begin{equation*}
%\label{fdk}
f_{\D_k}(x):=\frac{1}{D_k}\sum_{i\in \D_k}  f_i(x),
\end{equation*}
where $D_k:=| \D_k|$, $\D_k=\{l^k_1,...,l^k_{D_k}\}$, and  each $l^k_j \in \D_k$ takes the value $s \in \N:=\{1,...,N\}$ with probability $w_s$, i.e., 
\begin{equation} 
\label{Dk} 
P(l^k_j=s)=w_s, \quad s=1,...,N, \quad j=1,...,D_k.
\end{equation}
Although $D_k$  may be arbitrary, it is assumed that it is significantly smaller than $N_k$, and the common choice is $D_k=1$ for all $k$. The additional sampling rule within this paper is adapted to box constraints.
%It is modified in such way to take into account the structure of the projection operator \eqref{pibox}. 
 The additional sampling rule is also used to guide the acceptance of the candidate point. We will elaborate this in more detail in the next section. Finally, we emphasize that the additional sampling rule is constructed to determine if the sample size increase is needed, but allows an arbitrary increase. 
%provided that it does not exceed the full sample size. Although this may be beneficial since the theory covers all kinds of sampling strategies, having some guidance on the dynamics of increase would be useful, and thus it represents an interesting research direction which is beyond the scope of this paper.  

For simplicity, we start our analysis by observing non-negativity constraints, and later on we extend it to general box constraints by introducing some simple modifications within the algorithm and the convergence analysis.

\section{Nonnegativity Constraints: AS-NC method}
\label{sec3}
Within this section we consider an important special case of problem \eqref{problem}  given by 
\begin{equation}
    \label{problem-nc}
    \min_{x \geq 0} f(x),
\end{equation}
where the function $f$ is as in problem \eqref{problem} and inequalities $x \geq 0$ are component-wise.  Compared to the general box-constrained problem, this setting simplifies the structure of the feasible set, and we have  
%\begin{equation*}
%[\pi_S(y)]_i =
%\begin{cases}
%0, & y_i<0; \\
%y_i, & y_i\geq 0,
%\end{cases} 
%\end{equation*}
%in this case $[\pi_S(y)]_i$ is equal to: 1) $0$ if $y_i < 0$; 2)  $y_i$ if $y_i \geq 0$, 
%for $i=1,..,n$, i.e., $
$[\pi_S(y)]_i = \max\left\{ y_i,\, 0 \right\},i=1,..,n.  $ Since our direction will be of the form 
\begin{eqnarray}
    \label{pk}
    p_k = \pi_S(x_k-\nabla f_{\N_k}(x_k))-x_k,
\end{eqnarray}
we will distinguish two cases for each component $i \in \{1,...,n\}$: 
$$[p_k]_i=-[x_k]_i \quad \mbox{if } \quad [x_k]_i<[\nabla f_{\N_k}(x_k)]_i$$
and 
$$[p_k]_i=-[\nabla f_{\N_k}(x_k)]_i \quad \mbox{if} \quad [x_k]_i\geq[\nabla f_{\N_k}(x_k)]_i.$$ 
Let us denote by $I_{\N_k}$ an indicator vector of the event $x_k<\nabla f_{\N_k}(x_k), $ with inequality defined by components. More precisely,  for $i = 1, \ldots, n$ we have 
\begin{equation}
\label{INk}
[I_{\mathcal{N}_k}]_i =
\begin{cases}
1, & [x_k]_i < [\nabla f_{\mathcal{N}_k}(x_k)]_i \\
0, & [x_k]_i \geq  [\nabla f_{\mathcal{N}_k}(x_k)]_i.
\end{cases} 
\end{equation}
 Analogously, we define an indicator vector $I_{\D_k}$ of the event $x_k<\nabla f_{\D_k}(x_k)$  and 
\begin{equation}
    \label{rk}
    r_{\D_k}:=\|I_{\N_k}-I_{\D_k}\|.
\end{equation}
 Given that $ \mathcal{N}_k, \mathcal{D}_k $ and $ x_k $ are random, the values $ r_{\D_k} $ are also random and will be used 
to check the similarity of local cost functions in terms of the structure of the search direction $p_k$. Namely, notice that if $r_{\D_k}=0$ then  the structure of zero entries in $\pi_S(x_k-\nabla f_{\N_k}(x_k))$ is the same as for $\pi_S(x_k-\nabla f_{\D_k}(x_k))$.

\subsection{The Algorithm}

We state the algorithm for solving \eqref{problem-nc} as follows. 

\noindent {\bf Algorithm 1: AS-NC} (\textbf{A}dditional \textbf{S}ampling - \textbf{N}onnegativity \textbf{C}onstraints)
\label{AS-NC}
\begin{itemize}
\item[\textbf{S0}] \textit{Initialization.} Input: $x_0\geq 0, N_0 \in \mathbb{N}, \beta, c,c_1 \in (0,1), C>0$, $\{\varepsilon_k\}$ satisfying \eqref{sumable}. Set $k:=0$.
\item[\textbf{S1}] \textit{Subsampling.} If $N_k<N$, choose $\N_k$ such that \eqref{Nk} holds. Else, set $f_{\N_k}=f$.
\item[\textbf{S2}] \textit{Search direction.} Compute 
%$p_k$ via \eqref{pk}.
$ p_k = \pi_S(x_k-\nabla f_{\N_k}(x_k))-x_k. $
\item[\textbf{S3}] \textit{Step size.} Find  the smallest $j \in \mathbb{N}_0$ such that $t_k=\beta^j$ satisfies 
\begin{eqnarray}
    \label{ls}
    f_{\N_k}(x_k+t_k p_k)\leq f_{\N_k}(x_k)  + c_1 t_k (\nabla f_{\N_k}(x_k))^T p_k+\varepsilon_k.
\end{eqnarray}
Set $\bar{x}_k=x_k+t_k p_k$. 
\item[\textbf{S4}] \textit{Additional sampling.}\\ 
If $N_k=N$, set $x_{k+1}=\bar{x}_k$, $k=k+1$ and go to step S1.\\
Else choose $\D_k$ via \eqref{Dk} and compute \begin{eqnarray}
    \label{ddk}
    s_k = \pi_S(x_k-\nabla f_{\D_k}(x_k))-x_k
\end{eqnarray}
and $r_{\D_k} =  \|I_{\N_k}-I_{\D_k}\|.$  
\item[\textbf{S5}] \textit{Sample size update.}\\ If 
\begin{eqnarray}
    \label{ac}
   r_{\D_k}=0 \quad \mbox{and} \quad  f_{\D_k}(\bar{x}_k)\leq f_{\D_k}(x_k)  - c  \|   s_k\|^2+C\varepsilon_k,
\end{eqnarray}
$N_{k+1}=N_k$.\\ 
Else choose $N_{k+1} \in \{N_{k}+1,...,N\}$. 
\item[\textbf{S6}] \textit{Iterate update.} \\If 
$$f_{\D_k}(\bar{x}_k)\leq f_{\D_k}(x_k)  - c  \|   s_k\|^2+C\varepsilon_k$$ holds 
set $x_{k+1}=\bar{x}_k$. Else $x_{k+1}=x_k$. 
\item[\textbf{S7}] \textit{Counter update.} Set $k=k+1$ and go to Step \textbf{S1}.  
\end{itemize}

Notice that the algorithm can yield two types of scenarios: the Mini-batch (MB) scenario, where $N_k<N$ for all $k \in \mathbb{N}$, and the  Full sample (FS) scenario, where the full sample is eventually reached. Moreover, we say that AS-NC is in the MB phase at iteration $k$ if $N_k<N$. Otherwise, the full sample size is reached, i.e., if $N_k=N$, for some $ k $ then all further iterations have the same property and we say that we are in the FS phase.  In that case, the algorithm behaves as a deterministic projected gradient method. However, the sequence of iterates is still random due to sampling in the initial (MB) phase of the algorithm. 

In the MB phase, we have sampling at two steps of the algorithm: S1 and S4. Although we propose unbiased estimators \eqref{Nk} in step S1, the sampling used for $\N_k$ can in fact be arbitrary. This allows many strategies which can be very important from a practical point of view. Moreover, the choice of $\D_k$ may be modified as well, but it has to meet certain requirements  - it needs to be chosen independently of $N_k$ and it must allow positive probabilities for choosing each of the local cost function. Unbiased estimator is not essential for the convergence analysis. 

Notice that the sequence of iterates is feasible due to the construction of the algorithm. The search direction $p_k$ is a descent direction for $f_{\N_k}$ and feasible with respect to constraints, while backtracking line search retains feasibility.  The same type of direction is calculated in step S4, but with respect to $f_{\D_k}$, which is independent of $f_{\N_k}$. However, the Armijo-like condition is checked without performing any line search - it simply checks if the candidate point $\bar{x}_k$ is good enough for $f_{\D_k}$, which is an independent estimate of the objective function. Notice that in this check the constants $ c $ and $ C $ can be arbitrary small and large, respectively.  If the value of $ f_{\D_k}$ is good enough the candidate point is accepted at step S6. Otherwise, the step is rejected and the sample size $N_k$ is increased within step S5. The increase is arbitrary, as mentioned in Preliminaries. However, the sample size  $N_k$ can be increased also due to the different structure of the projection considering two approximate gradients $\nabla f_{\N_k}$ and $\nabla f_{\D_k}$, which discloses through $r_{\D_k}>0.$ Overall, the condition \eqref{ac} serves as the check of similarity of local cost functions and governs the sample size. Notice that calculating $r_{\D_k}$ does not yield additional costs since the structure observed in \eqref{INk} is needed for forming the projections as well. 

%Let us denote $\sigma$-algebra generated by $\{\N_0,\D_0,...,\N_{k-1}, \D_{k-1}\}$ with $\F_k$. Notice that $x_k$ is $\F_k$-measurable, but $p_k, t_k, f_{\N_k}, \nabla f_{\N_k}$ are not since they all depend on $\N_k$ which is random in the mini-batch (MB) mode -  when $N_k<N$. Now, let us denote $\sigma$-algebra generated by $\{\N_0,\D_0,...,\N_{k-1}, \D_{k-1}, \N_k\}$ with $\F_{k+1/2}$. Notice that $p_k, t_k, f_{\N_k}, \nabla f_{\N_k}, \bar{x}_k$ are $\F_{k+1/2}$-measurable. However, the acceptance of the candidate point depends on $\D_k$ in general, as well as the subsequent sample size. 

\subsection{Convergence analysis}
Within this section, we prove almost sure convergence of the proposed algorithm and analyze the complexity. 
We start the analysis by stating the following standard assumption. 

\begin{assumption}
\label{A1} Each function $f_i, i=1,...,N$ is continuously differentiable with $L$-Lispchitz continuous gradient and bounded from below by a constant $f_{low}$. 
\end{assumption}
As usual for additional sampling framework analysis, we proceed by dividing the set of all possible outcomes at iteration $k$ into two complementary subsets. Namely, let us denote  by $\D_k^+$ the subset of all possible outcomes of $\D_k$ at iteration $k$ for which the condition \eqref{ac} is satisfied, i.e.,
\begin{equation*} 
%\label{dkp}
    \D_k^+= \{\D_k \subset \N \; \vert \; r_{\D_k}=0, \quad   f_{\D_k}(\bar{x}_k)\leq f_{\D_k}(x_k)  - c  \|   s_k\|^2+C\varepsilon_k \}.
\end{equation*}
We denote the complementary subset of outcomes at iteration $k$ by
\begin{equation*} 
%\label{dkm}
    \D_k^{-}= \{\D_k \subset \N\; \vert \; r_{\D_k}>0 \quad \mbox{or} \quad   f_{\D_k}(\bar{x}_k)> f_{\D_k}(x_k)  - c  \|   s_k\|^2+C\varepsilon_k \}.
\end{equation*}

We begin our analysis with the following lemma, which basically describes the situation in which the full sample is not reached, based on choices of $\D_k$ that violate \eqref{ac}.   This lemma is conceptually aligned with Lemma 4.3 in \cite{ipas}, and the proof is the same as in \cite{ipas}, so it is omitted here.
%it here and delegate it to the Appendix for completeness. 
\begin{lemma}\label{Lemma-mbdkminus} [\protect{\cite[Lemma 4.3]{ipas}}]
    Suppose that Assumption A\ref{A1} holds. If $N_k<N$ for all $k \in \mathbb{N}$, then a.s. there exists $k_1 \in \mathbb{N}$ such that $\D_k^{-}=\emptyset$ for all $k \geq k_1$.
\end{lemma}
The following lemma states the well-known result for backtracking line search under the stated assumptions since, according to Theorem \ref{bmr} a),  there holds 
$$p_k^T \nabla f_{\N_k}(x_k) \leq -\|p_k\|^2.$$
\begin{lemma}\label{lemma-armijo}
Suppose that Assumption A\ref{A1} holds. Then the step size $t_k$ obtained from  step S3 satisfies 
\begin{align*}
t_k\geq t_{min}:=\min\left\{1,\dfrac{2\beta(1-c_1)}{ L}\right\}.
%\label{tkmin}
\end{align*}
\end{lemma}
Next, we prove the key result for the convergence analysis of AS-NC. Notice that \eqref{glavna1} is related to the original objective function and  $d$  defined as in  \eqref{pgd}, $ d(x):=\pi_{S}(x-\nabla f(x))-x$, regardless of the scenario (MB or FS). 

\begin{theorem} \label{teorema1}
    Suppose that Assumption A\ref{A1} holds. Then a.s. there exists a finite, random iteration $\tilde{k}$ such that for  all $k \geq \tilde{k}$ there holds 
    \begin{equation}
        \label{glavna1}
        f(x_{k+1})\leq f(x_k) -\bar{c} \|d(x_k)\|^2+\bar{C} \varepsilon_k,
    \end{equation}
    where $\bar{c}=\min \{c, c_1, 2c_1 (1-c_1) \beta / L\}$ and  $\bar{C}=\max \{1,C\}$.
\end{theorem}
    \begin{proof} Let us consider \text{the FS scenario first}. Then there exists a finite $\tilde{k}_1$ such that for all $k \geq \tilde{k}_1$ we operate with the true objective function $f$ and according to \eqref{ls} there holds 
    $$f(x_{k+1})\leq f(x_k)  + c_1 t_k (\nabla f(x_k))^T d(x_k)+\varepsilon_k\leq f(x_k)  - c_1 t_k \|d(x_k)\|^2+\varepsilon_k,$$
    where the last inequality comes from Theorem \ref{bmr} a). 
    Moreover,  Lemma \ref{lemma-armijo} implies that 
    $t_k\geq t_{min}$ 
    and thus we obtain 
        \begin{equation}
            \label{propfs}
            f(x_{k+1})\leq f(x_k)  - c_1 t_{min} \|d(x_k)\|^2+\varepsilon_k. 
        \end{equation}

    Now, let us observe \text{the MB scenario}. According to Lemma \ref{Lemma-mbdkminus} a.s. there exists some finite, random iteration $k_1$ such that $\D_k^{-}=\emptyset$ for all $k \geq k_1$. This means that the condition \eqref{ac} holds for all the local cost functions\footnote{Otherwise the set $\D_k^{-}$ would not be empty since one could form at least one possible $\D_k$ that violates \eqref{ac}, e.g., $\D_k=\{j_v,...,j_v\}$ where $j_v$ represents a local cost function that violates \eqref{ac}. }. 
    Therefore, for all $k \geq k_1$ and all $j \in \N$ there holds 
    $$f_{j}(\bar{x}_k)\leq f_{j}(x_k)  - c  \| s^j_k\|^2+C\varepsilon_k,$$
    where 
    $s^j_k := \pi_S(x_k-\nabla f_{j}(x_k))-x_k.$ Using the fact that in the considered scenario the candidate point would be accepted, i.e., $x_{k+1}=\bar{x}_k$, multiplying both sides with $w_j$ and summing up, we obtain 
    \begin{equation}
        \label{mbsum}
    f(x_{k+1})\leq f(x_k)  - c \sum_{j=1}^{N} w_j \| s^j_k\|^2+C\varepsilon_k.
    \end{equation}
    Let us consider the first condition of \eqref{ac}. Denote by $\A_{\N_k}$ the set of indices (components)  $i \in \{1,...,n\}$  such that $[I_{\N_k}]_i=1$, i.e., 
    \begin{equation}
        \label{aknk}
        \A_{\N_k}:=\{i \in \{1,...,n\} \; \vert \; [x_k]_i<[\nabla f_{\N_k}(x_k)]_i\}.
    \end{equation}
   Furthermore, using the similar arguments as for the second condition of \eqref{ac}, we conclude that $\D_k^{-}=\emptyset$ for all $k \geq k_1$ implies that $r_{\D_k}=0$ for all the singleton choices $\D_k=\{1\},..., \D_k=\{N\}$ for all $k \geq k_1$.  Having in mind the definition of $r_{\D_k}$ we conclude that for all $k \geq k_1$
   \begin{equation}
        \label{akovi}
        \A_{\N_k}=\A^1_{k}=...=\A^N_{k},
    \end{equation}
    where
   $\A^j_{k}:=\{i \in \{1,...,n\} \; \vert \; [x_k]_i<[\nabla f_{j}(x_k)]_i\}, j=1,...,N.$ 
   This further implies that all $k \geq k_1$, for all $i \in \A_{\N_k}$, for all $j \in \N$ there holds 
   $[x_k]_i<[\nabla f_{j}(x_k)]_i$ and thus 
   $$[x_k]_i=\sum_{j=1}^N w_j [x_k]_i<\sum_{j=1}^N w_j[\nabla f_{j}(x_k)]_i=[\nabla f (x_k)]_i, \quad \mbox{for all } \quad i \in \A_{\N_k}.$$
   Similarly, we obtain 
   $[x_k]_i\geq [\nabla f (x_k)]_i$ for all $i \notin \A_{\N_k}$ and due to \eqref{pgd} we conclude that the following holds for all $k \geq k_1$
   \begin{equation}
        \label{dgrad}
        [d(x_k)]_i=-[x_k]_i, \; \mbox{for all} \; i \in \A_{\N_k},  \; [d(x_k)]_i=-[\nabla f (x_k)]_i \; \mbox{for all} \; i \notin \A_{\N_k}.
    \end{equation}
   Now, let us estimate the norm of $d(x_k)$ for $k \geq k_1$ as follows 
   \begin{eqnarray} \label{p1} 
    \|d(x_k)\|^2 &=& \sum_{i=1}^{n} ([d(x_k)]_i)^2=\sum_{i\in \A_{\N_k}} ([d(x_k)]_i)^2+\sum_{i\notin \A_{\N_k}} ([d(x_k)]_i)^2\\\nonumber 
    & = & \sum_{i\in \A_{\N_k}} ([x_k]_i)^2+\sum_{i\notin \A_{\N_k}} ([\nabla f (x_k)]_i)^2. 
\end{eqnarray}
According to \eqref{akovi} we have for all $k \geq k_1$ and all $j \in \N$
$$[s^j_k]_i = [\pi_S(x_k-\nabla f_{j}(x_k))]_i-[x_k]_i=-[x_k]_i, \; i \in \A_{\N_k}$$
and thus for all $ i \in \A_{\N_k} $
\begin{equation}
     \label{p2} 
\sum_{j=1}^{N} w_j ([s^j_k]_i)^2 =\sum_{j=1}^{N} w_j ([x_k]_i)^2=([x_k]_i)^2,
\end{equation}
which further implies 
    \begin{equation}
     \label{p3} 
\sum_{i \in \A_{\N_k}} \sum_{j=1}^{N} w_j   ([s^j_k]_i)^2 =\sum_{i \in \A_{\N_k}}([x_k]_i)^2,
\end{equation}
Similarly, we conclude that for all $k \geq k_1$ and all $j \in \N$ there holds 
$$[s^j_k]_i =-[\nabla f_{j}(x_k)]_i, \; i \notin \A_{\N_k}$$
and  we conclude that for all $i \notin \A_{\N_k}$
 \begin{equation}
     \label{p4} 
([\nabla f (x_k)]_i)^2=(\sum_{j=1}^{N} w_j [\nabla f_j(x_k)]_i)^2\leq \sum_{j=1}^{N} w_j ([\nabla f_j(x_k)]_i)^2=\sum_{j=1}^{N} w_j ([s_k^j]_i)^2
\end{equation}
which further implies 
 \begin{equation}
     \label{p5} 
     \sum_{i\notin \A_{\N_k}} ([\nabla f (x_k)]_i)^2\leq \sum_{i\notin \A_{\N_k}} \sum_{j=1}^{N} w_j ([s_k^j]_i)^2.
     \end{equation}
Combining \eqref{p1}, \eqref{p3} and \eqref{p5} we obtain 
 \begin{eqnarray} \label{p6} 
    \|d(x_k)\|^2 &\leq & \sum_{i \in \A_{\N_k}} \sum_{j=1}^{N} w_j   ([s^j_k]_i)^2+\sum_{i\notin \A_{\N_k}} \sum_{j=1}^{N} w_j ([s_k^j]_i)^2\\\nonumber 
    &=& \sum_{j=1}^{N} w_j \left(\sum_{i \in \A_{\N_k}} ([s^j_k]_i)^2+\sum_{i\notin \A_{\N_k}}([s^j_k]_i)^2\right)\\\nonumber 
    &=& \sum_{j=1}^{N} w_j \|s_k^j\|^2. 
\end{eqnarray}
Combining this with \eqref{mbsum} we obtain for all $k \geq k_1$ 
\begin{equation}
        \label{mbsumf}
    f(x_{k+1})\leq f(x_k)  - c \|d(x_k)\|^2+C\varepsilon_k.
    \end{equation}
    
Taking into account both scenarios (FS and MB), i.e., \eqref{mbsumf} and \eqref{propfs}, we conclude the  proof with $\tilde{k} = k_1 $ in FS and $ \tilde{k} = \tilde{k}_1$ in MB case. 
\end{proof}

In order to obtain a.s. convergence, we impose the following assumption \cite{lsnmbb}. 
%which is essentially the same as Assumption 3 in  \cite{lsnmbb}. Let us denote by $\mathbb{E}_{FS}(\cdot) := \mathbb{E}(\cdot \mid FS)$ conditional expectation with respect to all the sample paths falling into the FS scenario. Analogously, we define $\mathbb{E}_{MB}(\cdot) := \mathbb{E}(\cdot \mid MB)$,  and the corresponding conditional probabilities  $\mathbb{P}_{FS}(\cdot)$ and $\mathbb{P}_{MB}(\cdot)$.

%\begin{assumption}
 %   \label{assNEW}
    
%There exists a constant $C_{b}$ such that 
%$$\mathbb{E}_{FS} ( \vert f(x_{\tilde{k}_1})\vert ) \leq C_{b}\quad \mbox{and}  \quad \mathbb{E}_{MB} ( \vert f(x_{k_1})\vert ) \leq C_{b},$$
%where $\tilde{k}_1 $ and $k_1$ are as in the proof of Theorem \ref{teorema1}. 
%\end{assumption}
\begin{assumption}
    \label{assNEW}   
There exists a constant $C_{b}$ such that 
$\mathbb{E} ( \vert f(x_{\tilde{k}})\vert ) \leq C_{b}$, where $\tilde{k} $ is specified  in  Theorem \ref{teorema1}.
\end{assumption}
 The above assumption is clearly fulfilled if the sequence $ \{f(x_k)\} $ is bound\-ed. Moreover, in the case of bounded iterates (e.g., compact feasible set as a special case in Section \ref{sec4}) the assumption holds for many objective functions.  But it also holds in more general situations as it allows the case when  
$f_{\mathcal{N}_{\tilde{k}}}(x_{\tilde{k}})$ is unbounded in general (for some sample paths), but the expectation over all possible sample paths is still bounded.
Let us denote by $\mathbb{E}_{FS}(\cdot):= \mathbb{E}(\cdot \mid FS)$ the conditional expectation concerning all the sample paths falling into the FS scenario. Analogously, we define $\mathbb{E}_{MB}(\cdot) := \mathbb{E}(\cdot \mid MB)$. It can be shown (see \cite{lsnmbb} e.g.) that Assumption A\ref{assNEW} implies 
\begin{equation}
    \label{cb12}
\mathbb{E}_{FS} ( \vert f(x_{\tilde{k}_1})\vert ) \leq C^{FS}_{b}\quad \mbox{and}  \quad \mathbb{E}_{MB} ( \vert f(x_{k_1})\vert ) \leq C^{MB}_{b},\end{equation}
for some constants $C^{FS}_{b}, C^{MB}_{b}$
where $\tilde{k}_1 $ and $k_1$ are as in the proof of Theorem \ref{teorema1}. 

Next, we state the main convergence result for AS-NC. 
\begin{theorem} \label{teorema2}
    Suppose that Assumptions A\ref{A1} and A\ref{assNEW} hold. Then a.s. every accumulation point of sequence $\{x_k\}_{k \in \mathbb{N}}$ generated by AS-NC is a stationary point of the problem \eqref{problem-nc}. 
\end{theorem}
\begin{proof}
%{\color{magenta} Ubaciti dokaz... Pokazati prvo da dk tezi nuli skoro sigurno (\cite{lsnmbb}). Zatim pokazati glavni rezultat (mozete pogledati \cite{BB})}
According to \eqref{glavna1} we have that a.s. 
$$f(x_{\tilde{k}+l})\leq f(x_{\tilde{k}}) -\bar{c} \sum_{j=0}^{l-1}\|d(x_{\tilde{k}+j})\|^2+\bar{C}  \sum_{j=0}^{l-1} \varepsilon_{\tilde{k}+j},$$
for any $l \in \mathbb{N}$. Applying the expectation and using Assumption A\ref{assNEW} together with sumability of $ \varepsilon_k $ given in \eqref{sumable}, by letting $l \to \infty$ we obtain 
$$\sum_{j=0}^{\infty}\mathbb{E}(\|d(x_{\tilde{k}+j})\|^2) <\infty. $$
Now, the  extended Markov's inequality and the Borel-Cantelli lemma (see e.g.  \cite{lsnmbb} for details), we conclude that 
\begin{equation}
   \mathbb{P}( \lim_{k\to\infty} d(x_k) = 0)=1.
\label{dknula}
\end{equation} 
Let $x^*$ be an arbitrary accumulation point of the sequence $\{x_k\}$, and let $K_0 \subset \mathbb{N}$ be a subsequence such that
$$ \lim_{k \in K_0} x_k = x^*.$$
Due to continuity of the gradient and the projection operator, from  \eqref{dknula} we conclude that a.s. 
$$
0  =\lim_{k \in K_0} d(x_k)=\lim_{k \in K_0}\left( \pi_S(x_k - \nabla f(x_k))-x_k\right)= \pi_S(x^* - \nabla f(x^*)) - x^*=d(x^*)$$
and by  Theorem~\ref{bmr} b) and the feasibility of the iterates, we conclude that $ x^* $ is a.s.  a stationary point of problem~\eqref{problem-nc}, which completes the proof. 
\end{proof}

Next, we analyze the complexity of the proposed method. The analysis combines techniques of \cite{SBNKNKJ}, \cite{grapiglia}, and \cite{lsnmbb}. We impose the assumption used in \cite{lsnmbb}. It states that the local cost functions are not homogeneous in the following sense.

\begin{assumption} \label{asswcc}
For each $k$ there exists at least one  function $f_i$ such that the condition \eqref{ac}
 \begin{eqnarray*}
    %\label{ac}
   r_{\D_k}=0 \quad \mbox{and} \quad  f_{\D_k}(\bar{x}_k)\leq f_{\D_k}(x_k)  - c  \|   s_k\|^2+C\varepsilon_k,
\end{eqnarray*}
is violated. 
\end{assumption}
 This assumption is likely to be satisfied in the case of data fitting if the data is heterogeneous or if the local cost functions  $ f_i $ are of a different type. In fact, considering, for instance, linear least squares problems, it can easily happen that a descent direction of one function is an ascent direction of another one.  

\begin{theorem} \label{teorema3}
    Suppose that Assumptions A\ref{A1}, A\ref{assNEW} and A\ref{asswcc} hold. Then the expected number of iterations to reach $\|d(x_k)\| < \nu $ is upper bounded  by 
    $$\hat{k}_{E}= \left\lceil \frac{N-1}{q} \right\rceil+\left\lceil \frac{C^{FS}_b-f_{low}+\bar{\varepsilon}}{\bar{c} \nu^2} \right\rceil,$$
    where $\bar{c}$ is as in Theorem \ref{teorema1},  $C_b^{FS}$  as in \eqref{cb12} and  $q=\min \{w_1,...,w_N\}^{N-1}$.
\end{theorem}

\begin{proof}
Assumption A\ref{asswcc} ensures that for every iteration $k$, there exists at least one function $f_i$ that violates the condition \eqref{ac}. Therefore, according to the distribution of $\D_k$ \eqref{Dk}, we conclude that  
$$
\mathbb{P}(\mathcal{D}_k \in \mathcal{D}^{-}_k) \geq \min \{w_1,...,w_N\}^{D_k}\geq  \min \{w_1,...,w_N\}^{N-1}=q. 
$$
Further, let us denote by $S_k$ a random variable that counts the number of increments of the sample size until iteration $k$. Notice that $S_k$ can be represented as 
$
S_k = I_1 + I_2 + \dots + I_k, 
$
where $I_k$ is an indicator variable, i.e., $ I_k = 1 $ if $ N_k > N_{k-1} $ and $ I_k = 0 $ otherwise. Furthermore, according to step S5 of AS-NC algorithm,   the increase of the sample size happens if and only if $\D_k \in \D^{-}_k$ and thus  $$\mathbb{E}(I_k) = P(I_k = 1) = P(\D_k \in \D^{-}_k) \geq q,$$ which further implies 
\begin{equation}
\mathbb{E}(S_k) \geq kq.
    \label{Sk}
\end{equation}
Let $ \tilde{N} $ represent the number of increments of the sample size needed to reach the full sample.{\footnote{For instance, if we set $ N_{k+1} = N_k + 1 $ at the end of step S5 of AS-NC, then $ \tilde{N} = N - N_0 $. See \cite{lsnmbb} and the text after Assumption 4 therein for further discussion on this topic. }}
Requiring $ \mathbb{E}(S_{\tilde{k}}) = \tilde{N} $ and using \eqref{Sk}, we conclude that the expected number of iterations to reach the full sample is bounded from above by $ \lceil \tilde{N}/q \rceil $ which can further be upper bounded by 
\begin{equation} \label{exp1}
   \left \lceil \frac{N-1}{q}\right\rceil.  
\end{equation} 

Furthermore,   let $\tilde{k}_1$ be the starting iteration of the FS phase. Then the decrease condition  \eqref{propfs} holds and according to Assumptions A\ref{A1}, A\ref{assNEW} and  \eqref{sumable} we conclude that for any $j \in \mathbb{N}$ we have 
\begin{equation} \label{exp2} 
\sum_{k=\tilde{k}_1}^{\tilde{k}_1+j} \mathbb{E}_{FS}(\|d(x_k)\|^2) \leq \frac{C^{FS}_b - f_{\text{low}} +  \bar{\varepsilon}}{\bar{c}}.
\end{equation}
Obviously, for each FS scenario there holds $\lim_{k \to \infty} \|d(x_k)\|=0$. Now, let us denote by $T$ the number of iterations   (counting from $\tilde{k}_1$) needed to reach  $ \|d(x_k)\| < \nu $. Then, we have  
$$
\sum_{k=\tilde{k}_1}^{\tilde{k}_1+T-1} \mathbb{E}_{FS}(\|d(x_k)\|^2) \geq \sum_{k=\tilde{k}_1}^{\tilde{k}_1+T-1} \nu^2 = T \cdot \nu^2
$$
and thus due to \eqref{exp2} we obtain 
$$T\leq \frac{C_b^{FS} - f_{\text{low}} +  \bar{\varepsilon}}{\nu^2 \bar{c}}.$$
Combining this with \eqref{exp1} we obtain the result. 
\end{proof}

{\bf{Remark 1. }} Notice that the expected complexity bound $\hat{k}_E$ is very conservative.  Since we observe the FS scenario in the previous theorem, instead of $\bar{c}$, we can use $c_1 t_{min}$. Moreover, instead of $N-1$, we can take $\tilde{N}$, which reveals the influence of the dynamics of increase used in step S5 of AS-NC to the complexity bound. Finally, $q$ in fact depends on the size of the additional sample $D_k$. Thus, setting  e.g. $D_k=1$ for each $k$ yields  $q=\min \{w_1,...,w_N\}$ which can be significantly larger than $q=\min \{w_1,...,w_N\}^{N-1}$.  The complexity bound is of order $ {\cal O}(\nu^{-2}) $ which corresponds to the deterministic case. If we take $N_k = N$ for all $k$, the bound becomes deterministic, i.e., the vicinity of a stationary point is reached after at most $\hat{k}= \left\lceil \frac{f(x_0)-f_{low}+\bar{\varepsilon}}{\bar{c} \nu^2} \right\rceil,$
iterations. In this case, Assumptions A\ref{assNEW} and A\ref{asswcc} are not relevant and the method coincides with the standard projected gradient method well known from the literature, yielding the same complexity bounds as in \cite{kompleksnost-det}.

We end this analysis by observing the strongly convex poblems. In that case, under Assumption A\ref{A1}, there exists a unique solution  $x^*$   of problem \eqref{problem-nc}.   
%Moreover, we can identify $f_{low}$ with $f^*:=f(x^*)$. 
\begin{theorem} \label{teorema4}
Suppose that Assumption A\ref{A1} holds and that the sequence $\{x_k\}_{k \in \mathbb{N}}$ generated by AS-NC is bounded. If the function $f$ is strongly convex, then $\lim_{k \to \infty} x_k=x^*$ a.s. 
\end{theorem}
\begin{proof} Bounded iterates imply Assumption A\ref{assNEW} and thus Theorem \ref{teorema2} implies that every accumulation point of the sequence $\{x_k\}_{k \in \mathbb{N}}$ is a stationary point of \eqref{problem-nc} a.s.. On the other hand, the strong convexity implies that $x^*$ is the unique stationary point of problem \eqref{problem-nc}. Therefore, we conclude that all accumulation points of the sequence $\{x_k\}_{k \in \mathbb{N}}$ are equal to $x^*$ a.s., which further implies that the whole sequence converges to $x^*$ a.s. This completes the proof.  
\end{proof}

\section{Box Constraints: AS-BOX method}
\label{sec4}
Within this section, we observe the general box-constrained problems  \eqref{problem}. The analysis is essentially the same as for the non-negativity constraints case, and we focus on the differences needed to extend the results. The main difference is in the projection form, which further influences the changes in the definition of $r_{\D_k}$. These are, in fact, the only two modifications with respect to the AS-NC algorithm, as will be stated in the sequel. The convergence analysis is completely the same, except for the proof of Theorem \ref{teorema1}, which needs to be adapted to the general case. We start by analyzing the projection operator and specifying the form of the search direction in this setting.

Considering the set $ S= \{x \in \mathbb{R}^n \; | \; l_i \leq x_i \leq u_i, \; i=1,...,n\}, $ 
%the projection operator is defined by \eqref{pibox} and IZBACILA SAM OVO 
the search direction \eqref{pk} is thus given by  
\begin{equation} \label{pkbox}
[p_k]_i =
\begin{cases}
l_i-[x_k]_i, & [x_k]_i - [\nabla f_{\mathcal{N}_k}(x_k)]_i<l_i, \\
-[\nabla f_{\mathcal{N}_k}(x_k)]_i, & l_i\leq [x_k]_i - [\nabla f_{\mathcal{N}_k}(x_k)]_i\leq u_i, \quad i = 1, \ldots, n\\
u_i-[x_k]_i, &[x_k]_i - [\nabla f_{\mathcal{N}_k}(x_k)]_i>u_i.
\end{cases} 
\end{equation}
Analogously to \eqref{INk} we define 
\begin{equation*}
[\tilde{I}_{\mathcal{N}_k}]_i =
\begin{cases}
1, & \text{if } [x_k]_i - [\nabla f_{\mathcal{N}_k}(x_k)]_i < l_i, \\
2, & \text{if } l_i \leq [x_k]_i - [\nabla f_{\mathcal{N}_k}(x_k)]_i \leq u_i,  \quad i = 1, \ldots, n\\
3, & \text{if } [x_k]_i - [\nabla f_{\mathcal{N}_k}(x_k)]_i > u_i,
\end{cases}
%\label{eq:INk}
\end{equation*}
and  the indicator vector $ \tilde{I}_{\mathcal{D}_k}$ accordingly. Then, the sparsity similarity  vector analogous to \eqref{rk} is defined as 
\begin{equation}
    \label{trk}
    \tilde{r}_{\mathcal{D}_k}=\| \tilde{I}_{\mathcal{N}_k}- \tilde{I}_{\mathcal{D}_k} \|.
\end{equation}

\subsection{The Algorithm}
The algorithm differs from AS-NC only in steps S2 and S4. We state it for completeness. 

\noindent {\bf Algorithm 2: AS-BOX} (\textbf{A}dditional \textbf{S}ampling - \textbf{BOX} constraints)
\label{AS-BOX}
\begin{itemize}
\item[\textbf{S0}] \textit{Initialization.} Input: $x_0 \in S, N_0 \in \mathbb{N}, \beta, c,c_1 \in (0,1), C>0$, $\{\varepsilon_k\}$ satisfying \eqref{sumable}. Set $k:=0$.
\item[\textbf{S1}] \textit{Subsampling.} If $N_k<N$, choose $\N_k$ via \eqref{Nk}. Else, set $f_{\N_k}=f$.
\item[\textbf{S2}] \textit{Search direction.} Compute  $p_k$ via \eqref{pkbox}.
\item[\textbf{S3}] \textit{Step size.} Find  the smallest $j \in \mathbb{N}_0$ such that $t_k=\beta^j$ satisfies 
\begin{eqnarray}
    \label{lsbox}
    f_{\N_k}(x_k+t_k p_k)\leq f_{\N_k}(x_k)  + c_1 t_k (\nabla f_{\N_k}(x_k))^T p_k+\varepsilon_k.
\end{eqnarray}
Set $\bar{x}_k=x_k+t_k p_k$. 
\item[\textbf{S4}] \textit{Additional sampling.}\\ 
If $N_k=N$, set $x_{k+1}=\bar{x}_k$, $k=k+1$ and go to step S1.\\
Else choose $\D_k$ via \eqref{Dk} and compute \begin{eqnarray}
    \label{ddkbox}
    s_k = \pi_S(x_k-\nabla f_{\D_k}(x_k))-x_k
\end{eqnarray}
and $\tilde{r}_{\D_k}$ defined by  \eqref{trk}. 
\item[\textbf{S5}] \textit{Sample size update.}\\ If 
\begin{eqnarray}
    \label{acbox}
   \tilde{r}_{\D_k}=0 \quad \mbox{and} \quad  f_{\D_k}(\bar{x}_k)\leq f_{\D_k}(x_k)  - c  \|   s_k\|^2+C\varepsilon_k,
\end{eqnarray}
$N_{k+1}=N_k$.\\ 
Else choose $N_{k+1} \in \{N_{k}+1,...,N\}$. 
\item[\textbf{S6}] \textit{Iterate update.} \\If 
$$f_{\D_k}(\bar{x}_k)\leq f_{\D_k}(x_k)  - c  \|   s_k\|^2+C\varepsilon_k$$ holds 
set $x_{k+1}=\bar{x}_k$.\\
Else set $x_{k+1}=x_k$. 
\item[\textbf{S7}] \textit{Counter update.} Set $k=k+1$ and go to Step \textbf{S1}.  
\end{itemize}

\subsection{Convergence analysis}
The convergence analysis is conducted under the same set of assumptions. Notice that the results of Lemma \ref{Lemma-mbdkminus} and \ref{lemma-armijo} also hold  for AS-BOX.  Now, we state the result analogous to Theorem \ref{teorema1}. Notice that $ \tilde{k} $ has the same role as in Theorem \ref{teorema1}.

\begin{theorem} \label{teorema1box}
    Suppose that Assumption A\ref{A1} holds.  Then a.s. there exists a finite, random iteration $\tilde{k}$ such that for  all $k \geq \tilde{k}$ there holds 
    \begin{equation*}
       % \label{glavna1box}
        f(x_{k+1})\leq f(x_k) -\bar{c} \|d(x_k)\|^2+\bar{C} \varepsilon_k,
    \end{equation*}
    where $\bar{c}=\min \{c, c_1, 2 c_1 (1-c_1) \beta / L\}$ and  $\bar{C}=\max \{1,C\}$.
\end{theorem}
    \begin{proof} The first part of the proof is completely the same as the proof of Theorem \ref{teorema1}. Consider first the FS scenario. Analogously as in the proof of Theorem \ref{teorema1} we derive the inequality 
      \begin{equation}
            \label{propfs1}
            f(x_{k+1})\leq f(x_k)  - c_1 t_{min} \|d(x_k)\|^2+\varepsilon_k. 
        \end{equation}
    In the MB case, proceeding as in Theorem \ref{teorema1}  we conclude that for all $k \geq k_1$ there holds 
    \begin{equation}
        \label{mbsumbox}
    f(x_{k+1})\leq f(x_k)  - c \sum_{j=1}^{N} w_j \| s^j_k\|^2+C\varepsilon_k.
    \end{equation} 
     Notice that $ k_1 $ is again defined in Lemma \ref{Lemma-mbdkminus}. Now, let us define the following sets of indices
\begin{align}
\mathcal{L}_{\mathcal{N}_k} &:= \left\{ i \in \{1, \ldots, n\} \;\middle|\; [x_k]_i - [\nabla f_{\mathcal{N}_k}(x_k)]_i < l_i \right\},  \nonumber\\
\mathcal{I}_{\mathcal{N}_k} &:= \left\{ i \in \{1, \ldots, n\} \;\middle|\; l_i \leq [x_k]_i - [\nabla f_{\mathcal{N}_k}(x_k)]_i \leq u_i \right\}, \label{Iset} \nonumber\\
\mathcal{U}_{\mathcal{N}_k} &:= \left\{ i \in \{1, \ldots, n\} \;\middle|\; [x_k]_i - [\nabla f_{\mathcal{N}_k}(x_k)]_i > u_i \right\}. \nonumber
\end{align}
Using the same arguments as in the proof of Theorem \ref{teorema1} we conclude that  for all $k \geq k_1$ we have 
\begin{equation}
    \label{equality-sets}
    \mathcal{L}_{\mathcal{N}_k} = \mathcal{L}^1_k = \cdots = \mathcal{L}^N_k, \quad
    \mathcal{I}_{\mathcal{N}_k} = \mathcal{I}^1_k = \cdots = \mathcal{I}^N_k, \quad
    \mathcal{U}_{\mathcal{N}_k} = \mathcal{U}^1_k = \cdots = \mathcal{U}^N_k,
\end{equation}
where for each $j = 1, \dots, N$, we define 
\begin{align*}
\mathcal{L}^j_k &:= \left\{ i \in \{1, \ldots, n\} \;\middle|\; [x_k]_i - [\nabla f_j(x_k)]_i < l_i \right\}, \\
\mathcal{I}^j_k &:= \left\{ i \in \{1, \ldots, n\} \;\middle|\; l_i \leq [x_k]_i - [\nabla f_j(x_k)]_i \leq u_i \right\}, \\
\mathcal{U}^j_k &:= \left\{ i \in \{1, \ldots, n\} \;\middle|\; [x_k]_i - [\nabla f_j(x_k)]_i > u_i \right\}.
\end{align*}
This further implies that for all $ k \geq k_1 $, for all $ i \in \mathcal{L}_{\mathcal{N}_k} $, for all $ j \in \mathcal{N} $, there holds
$[x_k]_i - [\nabla f_j(x_k)]_i < l_i$,  i.e., $[x_k]_i < [\nabla f_j(x_k)]_i + l_i,$ and thus
$$
[x_k]_i = \sum_{j=1}^N w_j [x_k]_i < \sum_{j=1}^N w_j ([\nabla f_j(x_k)]_i + l_i) = [\nabla f(x_k)]_i + l_i, \quad \mbox{for all} \quad  i \in \mathcal{L}_{\mathcal{N}_k}.
$$
Similarly, for all $ i \in \mathcal{U}_{\mathcal{N}_k} $, we obtain
$[x_k]_i - [\nabla f(x_k)]_i > u_i,$
and for all $ i \in \mathcal{I}_{\mathcal{N}_k} $
$l_i \leq [x_k]_i - [\nabla f(x_k)]_i \leq u_i.$
Therefore, for all $ k \geq k_1 $ we have 
\begin{equation}
[d(x_k)]_i =
\begin{cases}
l_i - [x_k]_i, & \text{if } i \in \mathcal{L}_{\mathcal{N}_k}, \\
- [\nabla f(x_k)]_i, & \text{if } i \in \mathcal{I}_{\mathcal{N}_k}, \\
u_i - [x_k]_i, & \text{if } i \in \mathcal{U}_{\mathcal{N}_k}.
\end{cases}
\label{dgrad-box-final}\nonumber
\end{equation}
Now, let us estimate the norm of $d(x_k)$ for $k \geq k_1$ as follows 
\begin{eqnarray} \label{p1-box} 
\|d(x_k)\|^2 &=& \sum_{i=1}^{n} ([d(x_k)]_i)^2 \\
&=& \sum_{i \in \mathcal{L}_{\mathcal{N}_k}} (\ell_i - [x_k]_i)^2 + \sum_{i \in \mathcal{I}_{\mathcal{N}_k}} ([\nabla f(x_k)]_i)^2 + \sum_{i \in \mathcal{U}_{\mathcal{N}_k}} (u_i - [x_k]_i)^2. \nonumber
\end{eqnarray}
Recalling the definition $s^j_k := \pi_S(x_k-\nabla f_{j}(x_k))-x_k$, due to \eqref{equality-sets} we obtain 
$$
[s_k^j]_i =
\begin{cases}
\ell_i - [x_k]_i, & \text{if } i \in \mathcal{L}_{\mathcal{N}_k}, \\
- [\nabla f_j(x_k)]_i, & \text{if } i \in \mathcal{I}_{\mathcal{N}_k}, \\
u_i - [x_k]_i, & \text{if } i \in \mathcal{U}_{\mathcal{N}_k},
\end{cases}
$$
 for all $ k \geq k_1 $ and all $ j \in \mathcal{N}$.
Hence, for $ i \in \mathcal{L}_{\mathcal{N}_k} $ and $ k \geq k_1 $ there  holds
$$\sum_{j=1}^N w_j ([s_k^j]_i)^2 = \sum_{j=1}^N w_j (l_i - [x_k]_i)^2 = (l_i - [x_k]_i)^2,$$
and thus
\begin{equation}
    \label{box-low}
\sum_{i \in \mathcal{L}_{\mathcal{N}_k}} \sum_{j=1}^N w_j ([s_k^j]_i)^2 = \sum_{i \in \mathcal{L}_{\mathcal{N}_k}} (l_i - [x_k]_i)^2.
\end{equation}
Similarly, for $ i \in \mathcal{U}_{\mathcal{N}_k} $, we have
$$
\sum_{j=1}^N w_j ([s_k^j]_i)^2 = \sum_{j=1}^N w_j (u_i - [x_k]_i)^2 = (u_i - [x_k]_i)^2,
$$
and hence
\begin{equation}
    \label{box-up}
\sum_{i \in \mathcal{U}_{\mathcal{N}_k}} \sum_{j=1}^N w_j ([s_k^j]_i)^2 = \sum_{i \in \mathcal{U}_{\mathcal{N}_k}} (u_i - [x_k]_i)^2.
\end{equation}
Now, for $ i \in \mathcal{I}_{\mathcal{N}_k} $,
%we use convexity of the function $ (\cdot)^2 $ to estimate
$$
([\nabla f(x_k)]_i)^2 = \left( \sum_{j=1}^N w_j [\nabla f_j(x_k)]_i \right)^2 \leq \sum_{j=1}^N w_j ([\nabla f_j(x_k)]_i)^2 = \sum_{j=1}^N w_j ([s_k^j]_i)^2,
$$
which further implies
\begin{equation}
\label{box-int}
\sum_{i \in \mathcal{I}_{\mathcal{N}_k}} ([\nabla f(x_k)]_i)^2 \leq \sum_{i \in \mathcal{I}_{\mathcal{N}_k}} \sum_{j=1}^N w_j ([s_k^j]_i)^2.
\end{equation}
Combining \eqref{p1-box}, \eqref{box-low}, \eqref{box-up}, and \eqref{box-int}, we obtain
\begin{align}
\label{ineq-box}
\|d(x_k)\|^2 &\leq \sum_{i \in \mathcal{L}_{\mathcal{N}_k}} \sum_{j=1}^N w_j ([s_k^j]_i)^2 +
\sum_{i \in \mathcal{I}_{\mathcal{N}_k}} \sum_{j=1}^N w_j ([s_k^j]_i)^2 +
\sum_{i \in \mathcal{U}_{\mathcal{N}_k}} \sum_{j=1}^N w_j ([s_k^j]_i)^2 \nonumber \\
&=  \sum_{i=1}^n \sum_{j=1}^N w_j  ([s_k^j]_i)^2 = \sum_{j=1}^N w_j \sum_{i=1}^n ([s_k^j]_i)^2 = \sum_{j=1}^N w_j \|s_k^j\|^2.
\end{align}
Furthermore, combining this with \eqref{mbsumbox} we obtain for all $k \geq k_1$ 
\begin{equation}
        \label{mbsumfbox}
    f(x_{k+1})\leq f(x_k)  - c \|d(x_k)\|^2+C\varepsilon_k.
    \end{equation}
    Finally, Taking into account both scenarios (FS and MB), i.e., \eqref{mbsumfbox} and \eqref{propfs1}, we conclude the proof. 
\end{proof}

The proofs of the following three main results for AS-BOX are  the same as for AS-NC, so we only provide statements for completeness. 

\begin{theorem} \label{teorema2box}
    Suppose that Assumptions A\ref{A1} and A\ref{assNEW} hold. Then a.s. every accumulation point of sequence $\{x_k\}_{k \in \mathbb{N}}$ generated by AS-BOX is a stationary point of problem \eqref{problem}. 
\end{theorem}

\begin{theorem} \label{teorema3box}
    Suppose that Assumptions A\ref{A1}, A\ref{assNEW} and A\ref{asswcc} hold. Then the expected number of iterations to reach $\|d(x_k)\| < \nu $ is upper bounded  by 
    $$\hat{k}_{E}=\left \lceil \frac{N-1}{q} \right\rceil+\left\lceil \frac{C^{FS}_b-f_{low}+ \bar{\varepsilon}}{\bar{c} \nu^2} \right\rceil,$$
    where $\bar{c}$ is as in Theorem \ref{teorema1},  $C_b^{FS}$  as in \eqref{cb12} and  $q=\min \{w_1,...,w_N\}^{N-1}$.
\end{theorem}

\begin{theorem} \label{teorema4box}
Suppose that Assumption A\ref{A1} holds and that the sequence $\{x_k\}_{k \in \mathbb{N}}$ generated by AS-BOX is bounded. If the function $f$ is strongly convex, then   $\lim_{k \to \infty} x_k=x^*$ a.s. 
\end{theorem}

\section{Numerical results} 
\label{sec5}
In this section, we present numerical experiments designed to evaluate the performance of the proposed Algorithm \text{AS-BOX} and to compare it with existing methods from the literature. In particular, we focus on a comparison with the stochastic gradient-based interior-point method (SIPM) developed in \cite{curtis}, which was designed for solving smooth optimization problems with box constraints. The second benchmark method we consider is  PSGM (Projected Stochastic Gradient Method) as in  \cite{curtis}. In each iteration of the PSGM a stochastic gradient is computed and projected onto the feasible region defined by the box constraints. 
All the parameters used in our work that are related to SIPM and PSGM are the same as those provided in the numerical results section of \cite{curtis}.

The experiments 
%were conducted on a representative set of test problems, 
comprise two binary classification models: 
1)  logistic regression (convex); 
2) a single-hidden-layer neural network with cross-entropy loss (nonconvex).  
The experiments were conducted on datasets from the LIBSVM repository, namely \textit{Mushrooms }($8124$ samples, $112$ features) and \textit{IJCNN1} ($49990$ samples, $22$ features). These datasets are widely used benchmarks due to their diversity in structure,
which enables a comprehensive evaluation of algorithmic performance under varying problem structures.  Labels were encoded in $\{-1, 1\}$ format for binary classification.

We generated $x_0$ for each problem with elements drawn from a uniform distribution over $\left[-0.01, 0.01\right ]$. 
In both experiments, for the AS-BOX algorithm we use the following parameters:  $D_k = 1$, $C = 1$, $\beta = 0.1$, $\eta = 10^{-4}$, $c = 10^{-4}$,  $\varepsilon_k = \frac{1}{k^{1.1}}$\footnote{{ We also tested several alternative schedules and observed qualitatively unchanged behavior, indicating robustness to the choice of $\varepsilon_k$.}}. The parameters given above were chosen following the standard settings used in prior work \cite{aspen, ipas,lsnmbb}.

\subsection{Logistic Regression Problem}
The first benchmark problem is given by 
$$
    \min_{x \in [-1,\,1]^n} 
    \frac{1}{N} \sum_{i=1}^{N} \log\!\big(1 + \exp(-b_i a_i^\top x)\big),
$$
where $(a_i,b_i)$ are the training samples, $b_i \in \{-1,1\}$ are binary labels. This model is convex and commonly used as a baseline for large-scale classification tasks. Box constraints $[-1,1]^n$ were imposed to conform with the setup in \cite{curtis} and to enable direct comparison with existing stochastic methods. We model the computational cost by $FEV_k$ – the number of scalar products required by the specified method to compute $x_k$, starting from the initial point $x_0$.

To evaluate the performance of the considered methods, we present: the distance between $x_k$ and the solution $x^*$ of the considered problem, i.e., $||x_k - x^*||$, against the computational cost measure $FEV_k$. The stopping criterion in all comparisons is a fixed budget of scalar products (FEV), so the methods are evaluated up to the same computational effort.

%Figures \ref{Mushroomslog} and \ref{ijcnn1log} present the comparison results of all three algorithms on the previously mentioned datasets, \textit{Mushrooms} and \textit{IJCNN1}, respectively.

\begin{figure}[htb!]
    \centering
    \includegraphics[width=\textwidth]{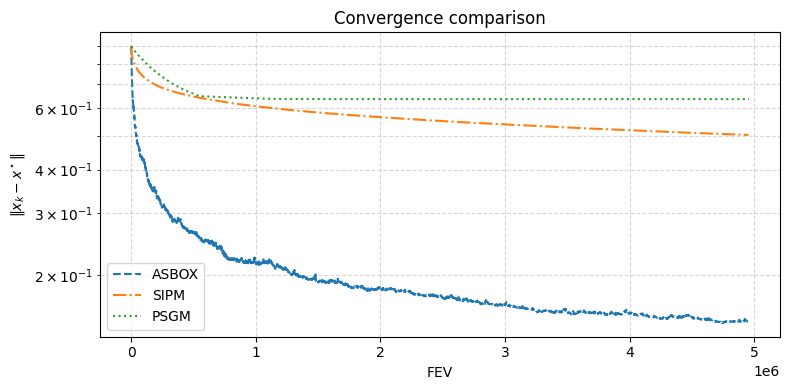}
    \caption{\footnotesize{ Distance to the solution $||x_k - x^*||$ versus $FEV_k$ for logistic regression on the \textit{Mushrooms} dataset. }}
     \label{Mushroomslog}
\end{figure}
In Figure \ref{Mushroomslog}, the comparison of the four algorithms (AS-BOX, SIPM, PSGM, and FULL SAMPLE (AS-BOX full sample, for all $k: N_k=N,$)) is demonstrated on the \textit{Mushrooms} dataset.
%for the logistic regression problem.
The graph shows the Euclidean distance $||x_k - x^*||$ as a function of the number of scalar products $FEV_k$, where AS-BOX achieves the fastest convergence toward the reference solution. SIPM shows a slower but steady decrease, while PSGM seems to be stagnating. On this dataset, the FULL SAMPLE ($N_k=N$) performs the worst overall, exhibiting the slowest convergence.

A similar behavior can be observed on the \textit{IJCNN1} dataset (Figure \ref{ijcnn1log}), where AS-BOX again achieves the best performance and reaches the smallest distance to the reference solution, attaining an accuracy level of $10^{-1}$ between $150{,}000$ and $200{,}000$ $FEV_k$, SIPM follows with moderate convergence, and PSGM remains the slowest method. On this dataset, the FULL SAMPLE ($N_k=N$) is slightly better than PSGM, but still clearly worse than AS-BOX and SIPM.
%, converging very slowly toward the optimal solution.
\begin{figure}[htb!]
    \centering
    \includegraphics[width=\textwidth]{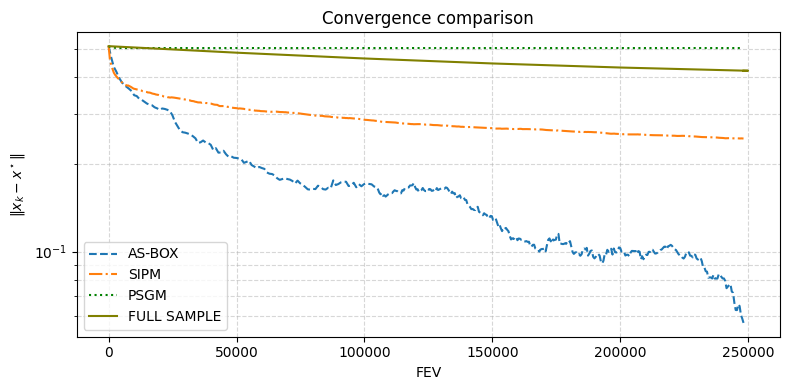}
    \caption{\footnotesize{Distance to the solution $||x_k - x^*||$ versus $FEV_k$ for logistic regression on the \textit{IJCNN1} dataset. }}
     \label{ijcnn1log}
\end{figure}

 In Figure \ref{comparisonfull} we report the evolution of the subsample size $N_k$ as a function of the computational budget $\mathrm{FEV}_k$ for AS-BOX on two datasets. The growth of $N_k$ is monotone and staircase-like - $N_k$ increases only when the acceptance test fails, and remains flat otherwise. On \textit{Mushrooms} ($m=8{,}124$), $N_k$ peaked at $168$ (about $2.1\%$ of the data); on \textit{IJCNN1} ($m=49{,}990$), it peaked at $504$ (about $1.0\%$). In neither case did the method reach the full sample size, indicating that AS-BOX makes steady progress without resorting to $N_k=N$, which underlies its computational efficiency at a fixed $\mathrm{FEV}_k$ budget.

\begin{figure}[ht!]
    \centering
    \begin{minipage}[c]{0.48\textwidth}
        \centering
        \includegraphics[width=\textwidth]{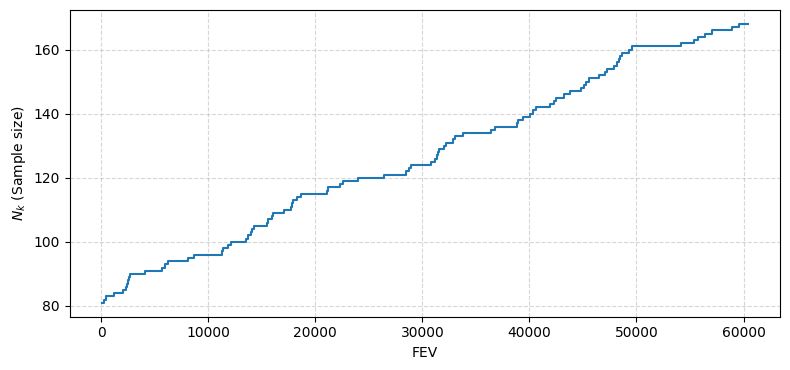}\\
        \footnotesize a)
    \end{minipage}
    \hfill
    \begin{minipage}[c]{0.48\textwidth}
        \centering
        \includegraphics[width=\textwidth]{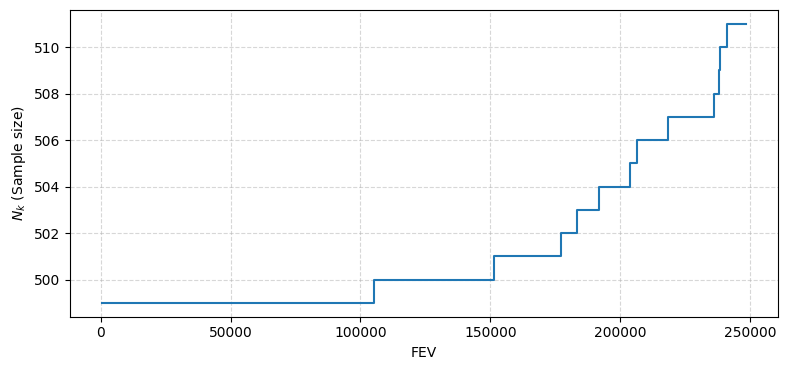}\\
        \footnotesize b)
    \end{minipage}
    \hfill
    \caption{\footnotesize{AS-BOX: evolution of the subsample size $N_k$ as a function of $FEV_k$. Part a): $N_k$ versus $FEV_k$ for the \textit{Mushrooms} dataset. Part b): $N_k$ versus $FEV_k$ for the \textit{IJCNN1} dataset.}}
    \label{comparisonfull}
\end{figure} 
 
\subsection{Neural Network with Cross-Entropy Loss}

The second problem considers training a single-hidden-layer neural network for binary classification. Let  $\tanh(\cdot)$ be the activation function in the hidden layer, while the output layer uses a sigmoid activation.  The network output is therefore given by
\[
    \sigma\big( W_2 \tanh(W_1 a + b_1) + b_2 \big),
\]
where $W_1, W_2$ are weight matrices and $b_1, b_2$ are bias vectors. The training objective is the average cross-entropy loss
\[
    \min_{x \in [-1,\,1]^d} 
    \frac{1}{N} \sum_{i=1}^N 
    \Big[ -y_i \log(\hat{y}_i) - (1-y_i)\log(1-\hat{y}_i) \Big],
\]
where $x$ collects all parameters $(W_1,W_2,b_1,b_2)$ and $d$ denotes the total number of network parameters. As in the logistic regression case, the parameters are constrained to lie within $[-1,1]$.

This problem is inherently nonconvex and poses a stronger challenge to optimization methods. Its inclusion in the test suite allows us to assess the robustness of AS-BOX when applied to neural network training under box constraints. Since the solution of such problem is not unique in general, we plot the optimality measure $||d(x_k)||$ 
against the computational cost measure $FEV_k$ to evaluate the performance of the considered methods.

%To evaluate the performance of the considered methods, we plot the optimality measure $||d(x_k)||$ of the considered problem
%against the computational cost measure $FEV_k$. 
%where $||d(x_k)||$ denotes the norm of the projected gradient mapping, i.e., the distance between the current point $x_k$ and its projected gradient step onto the feasible set. 
%This metric measures the violation of the first-order optimality conditions under box constraints: 
%when $||d(x_k)|| = 0$, the point $x_k$ satisfies the Karush–Kuhn–Tucker (KKT) conditions and is therefore stationary.

Figure \ref{comparison} part a) shows the cross-entropy loss trajectory on the \textit{Mushrooms} dataset. 
AS-BOX again outperforms its competitors, with the loss dropping from approximately 
$2\times10^{-1}$ to below $10^{-2}$ within $10^{5}$ evaluations. 
SIPM displays a slower descent, converging around $5\times10^{-2}$, 
whereas PSGM initially outperforms AS-BOX  but stagnates around 
$3\times10^{-2}$.  
The results suggest that AS-BOX maintains its advantage across problems of different structure.
%datasets of different sizes and separability properties.
Figure \ref{comparison} part b) reports the stationarity measure for the \textit{Mushrooms} dataset. 
Consistent with the loss plots, AS-BOX achieves the most significant reduction, descending from roughly 
$10^{-1}$ to about $10^{-2}$ and exhibiting a stable convergence pattern despite minor stochastic fluctuations. 
SIPM steadily decreases but remains above $3\times10^{-2}$ by the end, while PSGM flattens out near  $4\times10^{-2}$ early in the run. 

\begin{figure}[ht!]
    \centering
    \begin{minipage}[c]{0.48\textwidth}
        \centering
        \includegraphics[width=\textwidth]{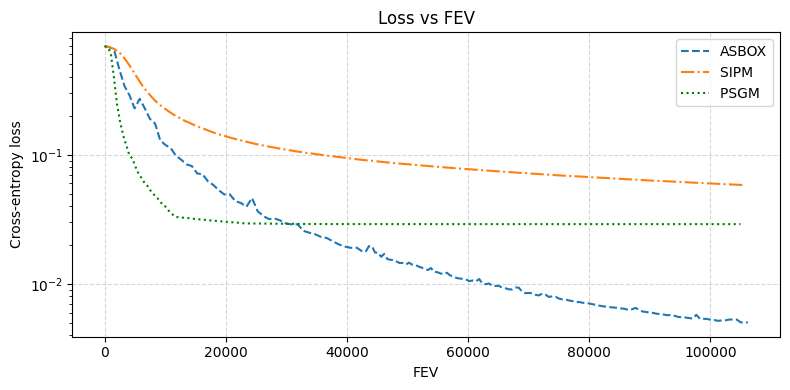}\\
        \footnotesize a)
    \end{minipage}
    \hfill
    \begin{minipage}[c]{0.48\textwidth}
        \centering
        \includegraphics[width=\textwidth]{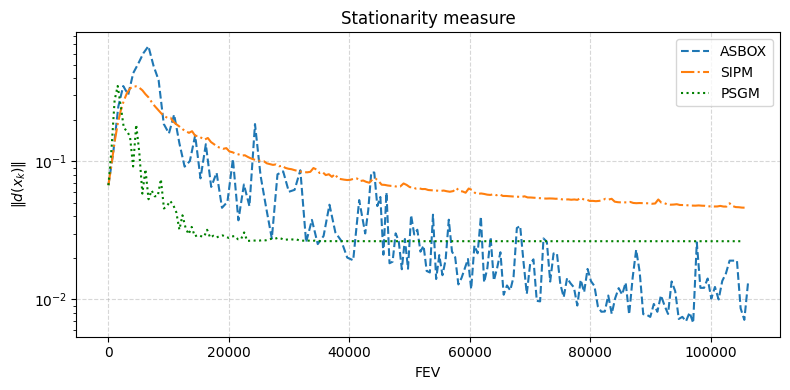}\\
        \footnotesize b)
    \end{minipage}
    \hfill
    \caption{\footnotesize{Part a): Cross-entropy loss versus $FEV_k$ for the \textit{Mushrooms} dataset. Part b): Stationarity measure $\|d(x_k)\|$ versus $FEV_k$ for the Mushrooms dataset. }}
     \label{comparison}
\end{figure} 

Next, Figure \ref{comparison1} part a) illustrates the evolution of the cross-entropy loss with respect to the number of function evaluations (FEV) for the \textit{IJCNN1} dataset. 
All algorithms start from a comparable initial loss of approximately $4.2\times10^{-1}$. 
The proposed AS-BOX method demonstrates the fastest and most consistent decrease, reaching a loss below $3.2\times10^{-1}$ after roughly $4\times10^{5}$ evaluations. The part b) on Figure \ref{comparison1} presents the stationarity measure $||d(x_k)||$ as a function of $FEV_k$ on the same dataset. 
This metric reflects how close the iterates are to satisfying the first-order optimality conditions under box constraints. 
AS-BOX exhibits the steepest decline, dropping below $2\times10^{-2}$ by the end of the run, indicating near-stationarity. 
SIPM converges more slowly, stabilizing around $4\times10^{-2}$, 
while PSGM decreases rapidly at first but stagnates near $4\times10^{-2}$. 
This highlights AS-BOX’s ability to achieve higher stationarity accuracy compared to the other methods. These results confirm the superior long-term convergence behavior of AS-BOX.

\begin{figure}[ht!]
    \centering
    \begin{minipage}[c]{0.48\textwidth}
        \centering
        \includegraphics[width=\textwidth]{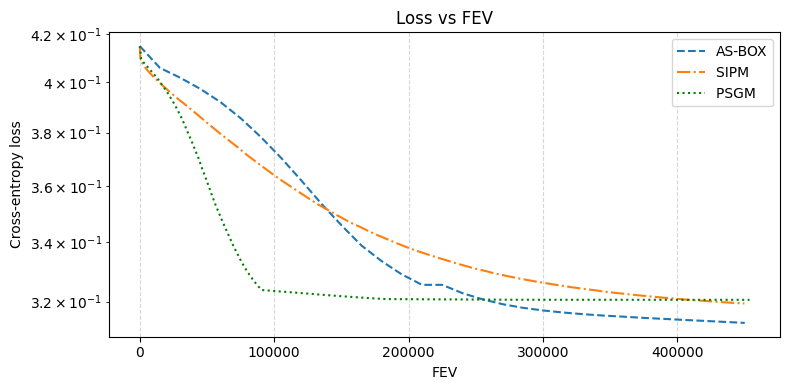}\\
        \footnotesize a)
    \end{minipage}
    \hfill
    \begin{minipage}[c]{0.48\textwidth}
        \centering
        \includegraphics[width=\textwidth]{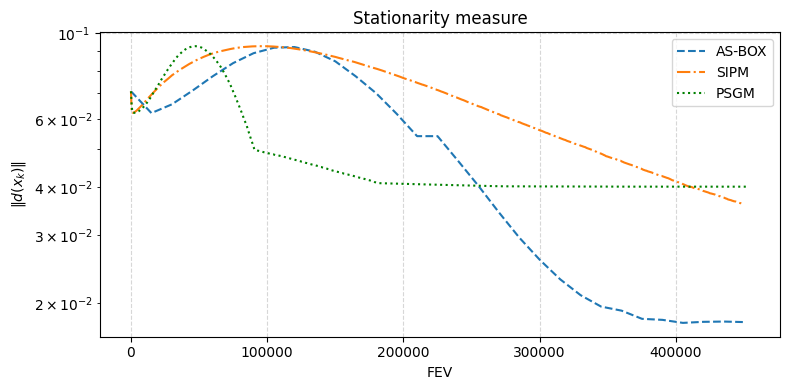}\\
        \footnotesize b)
    \end{minipage}
    \hfill
    \caption{\footnotesize{Part a): Cross-entropy loss versus $FEV_k$ for the \textit{IJCNN1} dataset. Part b): Stationarity measure $\|d(x_k)\|$ versus $FEV_k$ for the \textit{IJCNN1} dataset. }}
     \label{comparison1}
\end{figure} 

When comparing the results obtained on the \textit{IJCNN1} and \textit{Mushrooms} datasets, a consistent trend emerges: the AS-BOX algorithm achieves the biggest decrease in both cross-entropy loss and stationarity measure within the considered FEV budget, outperforming SIPM and PSGM across all experiments. However, the rate of convergence differs between the datasets due to their structural properties. The \textit{Mushrooms} dataset, being smaller and with more separable data, allows all algorithms to achieve lower loss values more quickly, with AS-BOX reaching near-optimal performance within $10^5$
  evaluations. In contrast, \textit{IJCNN1} is higher-dimensional and less separable, which slows down convergence for all methods; nevertheless, AS-BOX maintains a significant advantage over SIPM and PSGM, achieving roughly twice the reduction in stationarity by the end of the run.

\section{Conclusions}
\label{sec6}
A novel method (AS-BOX) for box-constrained weighted finite sum problems has been proposed. This method falls into the framework of stochastic projected gradient methods and uses non-monotone line search to adaptively determine the step size sequence, while retaining the feasibility of the iterates. The main novelty of AS-BOX lies in the adaptation of an additional sampling technique to box-constrained weighted finite-sum problems. Thus, the resulting method adaptively changes the sample size and conforms to different structures of the problems. AS-BOX also has a theoretical background - a.s. convergence is proved under a standard set of assumptions, without imposing the convexity. This makes it suitable for NN problems as well. Moreover, complexity analysis has been conducted as well, and a stronger convergence result is provided for strongly convex problems such as regularized logistic regression. Numerical study showed the efficiency of AS-BOX. 
%since it is comparable to state-of-the-art methods on both logistic regression and NN problems. 

Future work naturally tends to additional sampling methods for finite-sum problems with general, nonlinear equality and inequality constraints. 
%Moreover, the research directions may also include analysis of constrained infinite sum problems used to model  online training machine learning framework. Moreover, derivative-free methods would be of great interest, especially in NN applications, as well as methods for problems with non-differentiable objective functions appearing in Lasso regression e.g. Finally, finding  an optimal adaptive dynamics for the sample size increase would be extremely important from practical point of view.
%{\bf{Acknowledgement.}   }...
%We are grateful to the associate editor and two anonymous referees whose comments helped us improve the paper.}  

\vspace{1cm} 

{\bf{Funding.}  } 
N. Kreji\' c and N. Krklec Jerinki\' c are  supported by the Science Fund of the Republic of Serbia, Grant no. 7359, Project LASCADO.  T. Ostoji\' c is partially supported by the Ministry of Science, Technological Development and Innovation of the Republic of Serbia (Grants No. 451-03-136/2025-03/200156) and by the Faculty of Technical Sciences, University of Novi Sad through project “Scientific and Artistic Research Work of Researchers in Teaching and Associate Positions at the Faculty of Technical Sciences, University of Novi Sad 2025” (No. 01-50/295). N. Vu\v ci\' cevi\' c is partially supported by the Ministry of Science, Technological Development and Innovation of the Republic of Serbia (Grant no. 451-03-137/2025-03/ 200122)

%{\bf{Availability statement.}  } The datasets analyzed during the current study are available in ...
%the MNIST database of handwritten digits \cite{MNIST},  LIBSVM Data: Classification (Binary Class) \cite{SPLiADL} and UCI Machine Learning Repository \cite{MUSH}.
%All data generated or analyzed during this study are included in this published article.\\

%{\bf{Disclosure statement}}

%{\bf {Conflict of interest.}  } The authors declare no competing interests.


\begin{thebibliography}{99}

\bibitem{kompleksnost-det}
{\sc A. Beck,} (2017). 
First-order methods in optimization,
{\em Society for Industrial and Applied Mathematics.}

\bibitem{pregledni}
{\sc S. Bellavia, T. Bianconcini, N. Kreji\'c, \& B.  Morini,} (2021). 
Subsampled first-order optimization methods with applications in imaging, 
{\em Handbook of Mathematical Models and Algorithms in Computer Vision and Imaging.}

\bibitem{SBNKNKJ}
{\sc S. Bellavia, N. Kreji\' c,  \& N. Krklec Jerinki\' c,} (2020).
Subsampled Inexact Newton methods for minimizing large sums of convex function, 
{\em IMA Journal of Numerical Analysis 40(4), pp. 2309-2341.}

 \bibitem{SBNKNKJMR} 
 {\sc S. Bellavia, N. Kreji\' c, N. Krklec Jerinki\' c, \& M. Raydan}, (2024).
 SLiSeS: Subsampled Line Search Spectral Gradient Method for Finite Sums, 
 {\em Optimization Methods and Software, pp. 1–26.}
 
\bibitem{Birgin1}
{\sc E. G. Birgin, J. M. Martínez, \& M. Raydan,} (2000).
Nonmonotone Spectral Projected Gradients on Convex Sets.
{\em SIAM Journal on Optimization 10, pp. 1196-1211.}

\bibitem{byrd_adap}
{\sc R. Bollapragada, R. Byrd, \& J. Nocedal,} (2018). Adaptive sampling strategies for stochastic optimization. 
{\em SIAM Journal on Optimization, 28(4), 3312-3343.}

\bibitem{libsvm}
{\sc C. C. Chang, \& C. J. Lin, (2011).}
LIBSVM: A library for support vector machines,
{\em ACM Transactions on Intelligent Systems and Technology, 2:27:1--27:27.}

\bibitem{curtis}
{\sc F. E. Curtis, V. Kungurtsev, D. P. Robinson, \& Q. Wang,} (2025).
A stochastic-gradient-based interior-point algorithm for solving smooth bound-constrained optimization problems,
{\em SIAM Journal on Optimization, 35(2), 1030-1059.}

\bibitem{LSOS}
{\sc D. Di Serafino, N. Krejić, N. Krklec Jerinkić, \& M. Viola,} (2023).
LSOS: Line-search Second-Order Stochastic optimization methods for nonconvex finite sums. 
{\em Mathematics of Computation, 92(341), 1273-1299.}

%\bibitem{cutest}
%{\sc  N. I. M. Gould, D. Orban, \& P. L. Toint, } (2015).
%CUTEst: a Constrained and Unconstrained Testing Environment with safe threads for mathematical optimization. 
%{\em Comput Optim Appl 60, 545–557.}

\bibitem{grapiglia}
{\sc G. N. Grapiglia, \& Y. X. Yuan,} (2021).
On the complexity of an augmented Lagrangian method for nonconvex optimization,
{\em IMA Journal of Numerical Analysis, 41(2), 1546-1568.}

\bibitem{grippo}
{\sc L. Grippo, F. Lampariello, \& S. Lucidi,} (1986).
A nonmonotone line search technique for Newton's method,
{\em SIAM Journal on Numerical Analysis 23(4), pp. 707-716.}

\bibitem{he2024}
{\sc C. He, \& Z. Deng, } (2024). 
Stochastic interior-point methods for smooth conic optimization with applications. 
Available at: {\em arXiv:2412.12987.}

 \bibitem{UMU}{\sc T. Hastie, R. Tibshirani, \& J. Friedman,} (2009). 
 Elements of Statistical Learning, 
 {\em Springer. }

\bibitem{hu}
{\sc S. Huang, Z. Wan, Z., \& X. Chen,} (2015). 
A new nonmonotone line search technique for unconstrained optimization. 
{\em Numerical Algorithms, 68(4), 671-689.}

\bibitem{iusem1}
{\sc  A. N. Iusem,  A. Jofré,  R. I. Oliveira, \&  P. Thompson,} (2019). 
Variance-based extragradient methods with line search for stochastic variational inequalities. 
{\em SIAM Journal on Optimization, 29(1), 175-206.}  

\bibitem{iusem2}
{\sc  A. N. Iusem,  A. Jofré,  R. I. Oliveira, \&  P. Thompson,} (2017). 
Extragradient method with variance reduction for stochastic variational inequalities. {\em SIAM Journal on Optimization, 27(2), 686-724.} 
%\bibitem{dataset1} {\sc M. Lichman,} UCI machine learning repository (2013), {\em https://archive.ics. uci.edu/ml/index.php} 

%\bibitem{dataset2} {\sc Y. LeCun, C. Cortes, C. J. C. Burges,} The MNIST database of handwritten digits (1998), http://yann.lecun.com/exdb/mnist/

\bibitem{krejic4}
{\sc  N. Kreji\' c,  N. Krklec Jerinki\' c,  A. Martınez, \&  M. Yousefi,} (2024). 
A non-monotone trust-region method with noisy oracles and additional sampling. {\em Computational Optimization and Applications, 89(1), 247-278.}

\bibitem{NKNKJ2}
{\sc N. Kreji\' c, \& N. Krklec Jerinki\' c,} (2015). 
Nonmonotone line search methods with variable sample size, {\em Numererical Algorithms 68(4), pp. 711-739.}

%\bibitem{BB}
%{\sc  N. Krejić, \& N. Krklec Jerinkić,} (2019). 
%Spectral projected gradient method for stochastic optimization, 
%{\em Journal of Global Optimization, 73, 59-81.}

\bibitem{nasprvi}
{\sc N. Kreji\' c, N. Krklec Jerinki\' c, \& T. Ostoji\' c}, (2023).
An inexact restoration-nonsmooth algorithm with variable accuracy for stochastic nonsmooth convex optimization problems in machine learning and stochastic linear complementarity problems, 
{\em  Journal of Computational and Applied Mathematics, 423, 114943.}


\bibitem{nasdrugi}
{\sc N. Kreji\' c, N. Krklec Jerinki\' c, \& T. Ostoji\' c}, (2023).
Spectral Projected Subgradient Method for Nonsmooth Convex Optimization Problems, 
{\em Numerical Algorithms, , 93(1), 347-365.}

\bibitem{aspen}
{\sc N. Krejić, N. Krklec Jerinkić, T. Ostojić, \& N. Vučićević,} (2025).
ASPEN: An Additional Sampling Penalty Method for Finite-Sum Optimization Problems  with Nonlinear Equality Constraints, Available at: {\em arXiv: 2508.02299}.

\bibitem{ipas}
 {\sc N. Kreji\' c,  N. K. Jerinki\' c,  S. Rapaji\' c, \&  L. Rute\v si\' c,} (2025). IPAS: An Adaptive Sample Size Method for Weighted Finite Sum Problems with Linear Equality Constraints. Available at: {\em arXiv: 2504.19629.} 

\bibitem{proxbb} 
{\sc N. Krklec Jerinkić, F. Porta,  V. Ruggiero, \& I. Trombini,} (2025).
Variable metric proximal stochastic gradient methods with additional sampling, {\em Computational Optimization and Applications, https://doi.org/10.1007/s10589-025-00720-w.} 


\bibitem{lsnmbb} {\sc N. Krklec Jerinkić, V. Ruggiero, \& I. Trombini,} (2025). Spectral Stochastic Gradient Method with Additional Sampling for Finite and Infinite Sums, {\em Computational Optimization and Applications, 91 (2), 717–758.} 

%\bibitem{KLOS}
%{\sc N. Krejić, Z. Lužanin, Z. Ovcin, \& I. Stojkovska,} (2015).
%Descent direction method with line search for unconstrained optimization in noisy environment, 
%{\em Optimization Methods and Software 30(6), pp.
%1164-1184.}

\bibitem{li}
{\sc D. H. Li, \& M. Fukushima,} (2000).
A derivative-free line search and global convergence
of Broyden-like method for nonlinear equations, 
{\em Opt. Methods Software 13, pp. 181-201.}


\bibitem{nesterov1994interior}
{\sc Y. Nesterov, \& A. Nemirovskii,} (1994). 
Interior-point polynomial algorithms in convex programming. 
{\em Society for Industrial and Applied Mathematics.}

%\bibitem{SGD}
%{\sc H. Robbins, \& S. Monro,} (1951).
%A stochastic approximation method, 
%{\em The annals of mathematical statistics, pp. 400-407.}

%\bibitem{robbins}{\sc H. Robbins, \& D. Siegmund,} (1972).
%A convergence theorem for non negative almost supermartingales and some applications,
%{\em In: J. Multivariate Analysis, Vol. 2, pp. 267–275}

%\bibitem{rock}{\sc R. T. Rockafellar,}
%Convex Analysis, (1970).
%{\em Princeton University Press. }

\bibitem{roos1997interior}
{\sc C. Roos, \& J. P. Vial, }(1991). 
Interior Point Methods for Linear Programming: Theory and Practice. {\em North-Holland.}

\bibitem{wright1997primal}
{\sc S. J. Wright,} (1997). 
Primal-dual interior-point methods. 
{\em Society for Industrial and Applied Mathematics.}

\bibitem{wright}
{\sc  S. J. Wright, \&  B. Recht,}  (2022).
Optimization for data analysis. 
{\em Cambridge University Press.}

\bibitem{hz} {\sc H. C. Zhang, \& W. W. Hager,} (2004). A nonmonotone line search technique and its application to unconstrained optimization. {\em  SIAM Journal of Optimization, Vol. 14, No. 4, pp. 1043-1056.}
\end{thebibliography}
\end{document}